\documentclass[11pt]{article}
\usepackage{graphicx} 
\usepackage[top=2cm,bottom=2cm,left=2cm,right=2cm]{geometry}

\usepackage{palatino}
\usepackage{latexsym}

\usepackage[mathscr]{eucal}
\usepackage{amsmath}
\usepackage{amsthm}
\usepackage{amsfonts}
\usepackage{amssymb}
\usepackage{amscd}
\usepackage{color}
\usepackage{pifont}
\usepackage[makeroom]{cancel}
\usepackage[normalem]{ulem}
\usepackage{cite}
\usepackage{enumitem}
\usepackage{hyperref}

\theoremstyle{plain}
\newtheorem{theorem}{Theorem}[section]
\newtheorem{remark}[theorem]{Remark}
\newtheorem{lemma}[theorem]{Lemma}
\newtheorem{proposition}[theorem]{Proposition}
\newtheorem{definition}[theorem]{Definition}

\newcommand{\Z}{{\mathbb Z}}
\newcommand{\R}{{\mathbb R}}
\newcommand{\N}{{\mathbb N}}
\newcommand{\T}{{\mathbb T}}
\newcommand{\C}{{\mathbb C}}
\newcommand{\Sph}{{\mathbf S}} 

\def\({\left(}
\def\){\right)}
\def\<{\left\langle}
\def\>{\right\rangle}

\def\d{{\partial}}
\def\eps{\varepsilon}

\DeclareMathOperator{\RE}{Re}
\DeclareMathOperator{\IM}{Im}
\DeclareMathOperator{\Vect}{Vec}

\numberwithin{equation}{section}

\title{Small-time approximate controllability of the logarithmic Schr\"odinger equation} 
\author{Karine Beauchard, R\'emi Carles, Eugenio Pozzoli}

\date{}
\AtEndDocument{\bigskip{\footnotesize
  \noindent
  \textsc{Univ. Rennes, CNRS, IRMAR, UMR 6625, 35000 Rennes, France}\\
  \par
\noindent  \textit{E-mail address}:
  \texttt{\href{mailto:karine.beauchard@ens-rennes.fr}{karine.beauchard@ens-rennes.fr}} 
  \par
  \noindent  \textit{E-mail address}:
  \texttt{\href{mailto:remi.carles@math.cnrs.fr}{remi.carles@math.cnrs.fr}}
  \par
 \noindent  \textit{E-mail address}: 
 \texttt{\href{mailto:eugenio.pozzoli@univ-rennes.fr}{eugenio.pozzoli@univ-rennes.fr}}
}}
 
\begin{document}

\maketitle

\begin{abstract}
We consider  Schrödinger equations with logarithmic nonlinearity and
bilinear controls, posed on $\T^d$ or $\R^d$. We prove their
small-time global $L^2$-approximate controllability. The proof
consists in extending to this nonlinear framework the approach
introduced by the first and third authors in \cite{beauchard-pozzoli2}
to control the linear equation: it combines the small-time  
controllability of phases and gradient flows. Due to the
nonlinearity, the required estimates are more difficult to establish
than in the linear case. The proof here is inspired by WKB analysis. 

This is the first result of (small-time) global approximate
controllability, for nonlinear Schrödinger equations, with bilinear
controls. 
\end{abstract}



\section{Introduction}

\subsection{Models}

Let $d \in \N^*$ and $\lambda \in \R$. We consider logarithmic
Schr\"odinger equations (log-NLS) of the form
\begin{equation}\label{eq:schro}
\begin{cases}
\left( i\partial_t + \frac{1}{2} \Delta - V \right) \psi(t,x)=
\lambda \psi \log|\psi|^2(t,x) +\sum_{j=1}^m u_j(t)W_j(x) \psi(t,x), & (t,x) \in (0,T) \times M, \\
\psi(0,\cdot)=\psi_0,
\end{cases}
\end{equation}
where 
$M$ is either $\R^d$ or a smooth connected boundaryless Riemannian manifold,
$\Delta$ is the Laplace-Beltrami operator of $M$,
the functions $V, W_1,\dots,W_m: M \to \R$ are real valued potentials,
and the functions $u_1,\dots,u_m:(0,T) \to \R$ are real valued
controls, piecewise constant, $u_j\in
PWC([0,T],\R)$.
The time-dependent potential $\sum_{j=1}^m u_j(t)W_j(x)$ is possibly unbounded on $L^2(M,\C)$. For a time-dependent function 
$u=(u_1,\dots,u_m)$, and an initial state $\psi_0$
in the unitary sphere 
\begin{equation} \label{def:S}
\Sph:=\{\psi\in L^2(M,\C)\, ;\, \|\psi\|_{L^2(M)}=1\}.
\end{equation}
then $\psi(t;u,\psi_0)$ denotes - when it is well defined - the solution of \eqref{eq:schro}.
 In this article, we study in particular two examples of equations of the form \eqref{eq:schro}.

\medskip

\paragraph{An equation posed on $M=\T^d=\R^d/(2\pi\Z)^d$.}

The first example is the equation 
\begin{equation}\label{logNLS_tore}
\begin{cases}
\left( i\partial_t + \frac{1}{2} \Delta -V \right) \psi=
\lambda \psi \log|\psi|^2 +\sum\limits_{j=1}^d \( u_{2j-1}(t)\sin
  +u_{2j}(t)\cos)\langle b_j , x\rangle \)\psi,\quad
  (t,x)\in(0,T)\times\T^d,\\ 
\psi(0,\cdot)=\psi_0, &
\end{cases}
\end{equation}
where 
$V\in L^\infty(\T^d,\R)$, and 
\begin{equation} \label{Def:ej}
b_1=(1,0,\dots,0),\quad b_2=(0,1,\dots,0),\quad \dots, \quad b_{d-1}=(0,\dots,1,0),\quad b_d=(1,\dots,1).
\end{equation}

\paragraph{An equation posed on $M=\R^d$.}

The second example is the equation
\begin{equation}\label{logNLS+}
\begin{cases}
 \left( i\partial_t+ \frac{1}{2} \Delta - V \right)\psi
=\lambda \psi \log|\psi|^2
+\left( \sum\limits_{j=1}^d u_{j}(t)x_j
+u_{d+1}(t)e^{-|x|^2/2}  \right)\psi, \quad(t,x) \in (0,T) \times \R^d, \\
 \psi(0,\cdot)=\psi_0,  
\end{cases}
\end{equation}
where $V=V(x)$ is smooth, real-valued and at most quadratic, in the sense that
\begin{equation} \label{Hyp:V_transp}
V \in C^{\infty}(\R^d,\R) \text{ and } \forall \beta \in \N^d \text{
  with } |\beta|\ge 2, \ \partial_x^{\beta} V \in L^{\infty}(\R^d). 
\end{equation}
We will say that a (smooth) function is at most linear
if \eqref{Hyp:V_transp} holds for all $\beta\in \N^d\setminus\{0\}$. 
\medskip

In this article, we use piecewise constant controls $u$, then the solutions of \eqref{logNLS_tore} and \eqref{logNLS+} are well defined and $\psi(.;u,\psi_0) \in C^0([0,T],\Sph )$ (see Section~\ref{App:WP}).

\subsection{Bibliographical comments about bilinear control} 

The mathematical bilinear control theory of Schrödinger PDEs as
\eqref{eq:schro} has undergone a vast development in the last two
decades. Such theoretical problems find their origins in applications
of quantum control to physics and chemistry (e.g. absorption
spectroscopy) \cite{glaser}, or computer science (e.g. quantum
computation) \cite{preskill}. 

\subsubsection{Exact controllability}

The wavefunction is defined up to global phases (the state
$e^{i\theta}\psi_1$ for some constant $\theta\in \R$, is physically
the same as $\psi_1$) thus we can adopt the following definition. 

\begin{definition}[Exact controllability]
Let $\mathcal{H}$ be a subspace of $L^2(M,\C)$. We say that \eqref{eq:schro} is 
exactly controllable in $\mathcal{H}$ 
if, for every 
$\psi_0, \psi_1 \in \Sph \cap \mathcal{H} $ 
there exist 
a time $T>0$,
a global phase $\theta \in [0,2\pi)$
and a 
control $u:(0,T)\to\R$ such that 
$\psi( T;u,\psi_0)= \psi_1 e^{i\theta}$.
\end{definition}

In the linear case (i.e., $\lambda=0$ in \eqref{eq:schro}), by the seminal work \cite{BMS}, 
if the drift $i(\Delta-V)$ generates a group of bounded operators on
a Hilbert space $\mathcal{H}$, on which the control operators $W_j$
are bounded, then 
the equation \eqref{eq:schro} is not exactly controllable in
$\mathcal{H}$ (with controls $u \in L^{1}_{\rm loc}(\R,\R^m)$), because the reachable set has empty interior in $\Sph  \cap \mathcal{H}$  (see also \cite{Chambrion-Caponigro-Boussaid-2020} for recent developments). 

This topological obstruction to exact controllability persists for
Schrödinger equations with polynomial state nonlinearities \textit{à
  la} Gross-Pitaevskii (see
\cite{chambrion-laurent,chambrion-laurent2}). For logarithmic
nonlinearities as in \eqref{eq:schro}, the question is open. One might
expect to answer it by combining the techniques of 
\cite[Lemma~3.1]{chambrion-laurent2}, where the presence of an
$L^2$-eigenbasis (given by Hermite function when $V$ is the harmonic
potential) is crucial in order to uncouple space and time variables, with
\cite[Sections 3.2 and 3.3]{Hayashi-p}, where localizing arguments
make it possible to actually use Duhamel's formula, an aspect which appears to
be new for the logarithmic nonlinearity.

\medskip

To circumvent this topological obstruction, the exact controllability has been studied in spaces on which the $W_j$'s are not bounded. Such exact controllability results were proved for 1D-Schrödinger equations on an interval (see \cite{beauchard1,beauchard-coron,nersesyan-nersisyan,morancey-polarizability,bournissou} for linear equations and 
\cite{beauchard-laurent,beauchard-lange-teismann,duca-nersesyan2} for polynomial nonlinearities). These local controllability results are proved by linear test; this technique is not suitable 
for the multi-D case with non-$C^1$ nonlinearity, studied in this article.

\subsubsection{Approximate controllability (in potentially large time)}

\begin{definition}[Approximate controllability]
We say that \eqref{eq:schro} is 
$L^2$-approximately controllable
if, for every 
$\psi_0, \psi_1 \in \Sph $ and  $\varepsilon >0$, 
there exist 
a time $T>0$,
a global phase $\theta \in [0,2\pi)$
and a 
control $u\in PWC([0,T],\R^m)$ such that 
$\| \psi( T;u,\psi_0) - \psi_1 e^{i\theta} \|_{L^2} <\varepsilon$.
\end{definition}

Several methods were developed to prove the (large time) approximate controllability of linear Schrödinger equations with bilinear control, when the drift $-\Delta+V$ has discrete spectrum:
\begin{itemize}[topsep=0cm, parsep=0cm, itemsep=0cm]
\item  Applying the control theory for ODEs to the Galerkin approximations (i.e., the projections of the system onto
finite-dimensional eigenspaces of the drift) and estimating the error
between them and the solution of the PDE \cite{BCMS,MS-generic}; such techniques prove the large-time $L^2$-approximate controllability of  \eqref{logNLS_tore} with $d=1$ and $V,\lambda=0$  \cite{BCCS},
\item Finite-dimensional approximations and periodic control laws \cite{ervedoza},
\item Adiabatic approximation \cite{chittaro},
\item Lyapunov stabilization \cite{nerse,nersesyan}.
\end{itemize}

\medskip

Some large-time approximate controllability properties (for particular initial and final data) were also previously established for other PDEs, with bilinear control, e.g.: 
\begin{itemize}[topsep=0cm, parsep=0cm, itemsep=0cm]
\item 1D-linear wave equations in \cite{BMS}, via non-harmonic Fourier analysis,
\item 1D non-linear (resp. linear) heat equations in \cite{cannarsa-reactiondiffusion} (resp. \cite{cannarsa-khapalov}); the method employs shifting the points of sign change by making use of a finite sequence of initial-value pure diffusion problems,
\item 1D non-linear (resp. linear) wave equations in \cite{khapalov-wave1} (resp. \cite{khapalov-wave2}), via controlling a finite number of modes individually and subsequently in time.
\end{itemize}
In \cite{cannarsa-reactiondiffusion,cannarsa-khapalov,khapalov-wave1,khapalov-wave2} the control depends on both time and space variables.

\subsubsection{Obstructions to the small-time approximate controllability}

\begin{definition}
\label{def:STAC}
We say that \eqref{eq:schro} is \textbf{small-time $L^2$-approximately controllable} ($L^2$-STAC)  if, for every 
$\psi_0, \psi_1\in \Sph $ and
$\varepsilon >0$, there exist 
a time $T \in[0,\varepsilon]$,
a global phase $\theta \in [0,2\pi)$
and 
a control $u\in PWC([0, T ], \R^m)$ such that 
$\|\psi( T;u,\psi_0)-e^{i\theta} \psi_1\|_{L^2}<\varepsilon$.
\end{definition}
The small-time controllability has particularly relevant physical implications, both from a fundamental viewpoint and for technological applications. As a matter of fact, quantum systems, once engineered, suffer of very short lifespan before decaying (e.g., through spontaneous photon emissions) and losing their non-classical properties (such as superposition). The capability of controlling them in a minimal time is in fact an open challenge also in physics (see, e.g., the pioneering work \cite{khaneja-brockett-glaser} on the minimal control-time for spin systems). 

There exist examples of linear Schrödinger equations of the form \eqref{eq:schro} (with $\lambda=0$) which are $L^2$-approximately controllable in large times but not in small times. This obstruction happens e.g. when $M=\R^d$, in the presence of a sub-quadratic drift potential $V$ and quadratic or linear control potentials $W_j$, because Gaussian states are preserved, at least for small times \cite{beauchard-coron-teismann,beauchard-coron-teismann2} (see also \cite{obstruction-ivan} for different semi-classical obstructions). 
The extension of these results to the logarithmic Schrödinger equation is treated in this article (see Theorem \ref{Thm:negatif}).

\subsubsection{Small-time approximate controllability}

Recently, the first examples of small-time approximately controllable linear equations \eqref{eq:schro} (with $\lambda=0$) were given in \cite{beauchard-pozzoli} by the first and third authors: those systems correspond to 
$M=\R^d$, 
multi-input $m=2$, 
$W_1(x)=|x|^2$
and generic $W_2 \in L^{\infty}(\R^d)$. 
The control on the frequency of the quadratic potential $W_1$ permits to construct solutions that evolve approximately along specific diffeomorphisms, namely, space-dilations. Once we have access to space-dilations, we can exploit the scaling of the equation posed on $\R^d$ (with $u_2=0$) to generate time-contractions. In this way, we built on previous results of large-time control, and obtained small-time control. 

\medskip

In \cite{beauchard-pozzoli2}, we introduced a new method to prove the
small-time global approximate controllability of linear Schrödinger
equations, that does not require discrete spectrum. It applies to
equations \eqref{logNLS_tore} and \eqref{logNLS+} with $\lambda=0$.
The proof strategy will be explained later (see section
\ref{sec:Strategy}) because the purpose of this article is to extend
it to the case of the nonlinear equations \eqref{logNLS_tore} and
\eqref{logNLS+} with $\lambda\neq 0$. The small time asymptotics makes
it easier to handle the nonlinearity.

A key ingredient of this strategy is the small-time approximate
controllability of phases, which refers to the possibility, for any
initial condition $\psi_0 \in \Sph $ and phase $\varphi \in L^2(M,\R)$
to approximately reach, and in arbitrarily small times, the state
$e^{i \varphi}\psi_0 $. This property, introduced in
\cite{duca-nersesyan}, has been proved with a polynomial (instead of
logarithmic) nonlinearity in  \cite{duca-nersesyan} for system
\eqref{logNLS_tore}  and in \cite{duca-pozzoli} for system
\eqref{logNLS+}. The proof relies on Lie bracket techniques and a
saturation argument (introduced in the pioneering articles
\cite{agrachev-sarychev,agrachev2} on the linear control of
Navier-Stokes systems). The small-time controllability of phases is a
consequence of the density in $L^2(M,\R)$ of a particular functional
subspace of ${\rm Lie}\{\Delta-V,W_1,\dots,W_m\}$, shared by both
systems \eqref{logNLS_tore} and \eqref{logNLS+}. For appropriate
systems, it implies the small-time approximate controllability between
particular eigenstates \cite{Boscain-2024,chambrion-pozzoli}. In the
present article, we extend the small-time controllability of phases to
the logarithmic Schrödinger equation \eqref{eq:schro} (see
Theorem~\ref{thm:phase}).

\medskip

Small-time approximate controllability properties (for particular initial and final data) were also recently obtained for different PDEs, using the Agrachev-Sarychev saturation argument:
\begin{itemize}[topsep=0cm, parsep=0cm, itemsep=0cm]
\item Semiclassical Gross-Pitaevskii equations in \cite{coron-xiang-zhang},
\item Linear wave equations in \cite{pozzoli},
\item Polynomial nonlinear heat equations in \cite{duca-pozzoli-urbani}.
\end{itemize}


\subsubsection{Background on the logarithmic Schr\"odinger equation}

In the absence of external potential ($V=W_j=0$ in \eqref{eq:schro}),
the logarithmic Schr\"odinger equation was introduced in the context
of quantum mechanics in \cite{BiMy76}, because this logarithmic
nonlinearity is the only one, in the context of Schr\"odinger
equation, leading to the following tensorization principle: 
\begin{equation*}
  \psi_0(x) =\prod_{j=1}^d \psi_{0j}(x_j)
  \qquad \Rightarrow \qquad
   \psi(t,x) =\prod_{j=1}^d \psi_{j}(t,x_j),
\end{equation*}
where each $\psi_j$ solves a one-dimensional logarithmic Schrodinger equation with initial condition $\psi_{0j}$.
This property remains for potentials that decouple variables i.e. $V(x) = \sum_{j=1}^d V_j(x_j)$, $W_j(x)=W_j(x_j)$.
This tensorization principle would be an obstruction to controllability.
Here, we prevent it with the potentials $\sin \langle b_d , x \rangle$, $\cos \langle b_d , x \rangle$ in \eqref{logNLS_tore} and $e^{-|x|^2/2}$ in \eqref{logNLS+}.

The mathematical analysis of this equation goes back to
\cite{cazenave-haraux}, where the existence and uniqueness of
solutions are established, in the case $\lambda<0$. In particular,
uniqueness relies on a 
remarkable identity recalled in 
Lemma~\ref{lem:CH} (see also \cite{Hayashi-p} for recent improvements). As pointed out in \cite{BiMy79}, in the case
$V=W_j=0$ and $\lambda<0$, \eqref{eq:schro} possesses an explicit
solitary wave, whose profile (in space) is exactly a Gaussian. The
orbital stability of this object was proven in \cite{Caz83} in the radial
setting, and in \cite{Ar16} in the general case. 

The mathematical analysis of \eqref{eq:schro} has recently seen a resurgence of interest. The Cauchy problem was
revisited in \cite{GuLoNi10}, with no restriction on the sign of
$\lambda$, then in \cite{carles-gallagher} with a 
larger space for the initial data  (in particular, 
finite time blow up never occurs). These existence results were 
further refined in \cite{CHO24,HaOz25}, as we will see in
Section~\ref{sec:WP}. The dynamical properties of the logarithmic
Schr\"odinger equation on $\R^d$
turn out to be rather unexpected: a non-standard dispersive rate along
with a universal asymptotic profile (leading to an explicit growth
of Sobolev norms) were proven in the case
$\lambda>0$ in \cite{carles-gallagher}, while in the case $\lambda<0$,
a surprising superposition principle was justified in \cite{FeDCDS}, and the
existence of multisolitons and multibreathers was established in \cite{FeAIHP}. 

After the initial introduction in \cite{BiMy76}, the logarithmic
Schr\"odinger equation has been considered in several 
 physical models, for instance in
quantum mechanics \cite{yasue},
quantum optics \cite{hansson,KEB00,buljan}, nuclear
physics \cite{Hef85}, Bohmian mechanics \cite{DMFGL03}, effective
quantum gravity \cite{Zlo10}, and  Bose-Einstein condensation
\cite{BEC}. The presence of an external potential like in the present
article appears for instance in \cite{Bouharia2015}, in the case where
$V$ is a harmonic potential.
\smallbreak

On a mathematical level, a specificity of the logarithmic
Schr\"odinger equation is that the Cauchy problem needs a special
approach, as the nonlinearity $z\mapsto z\log|z|^2$ is not locally
Lipschitzian. The now 
standard strategy based on a fixed point argument involving Strichartz
estimates (see e.g. \cite{C-2003,TT-2006}) does not seem efficient in
the Cauchy problem \eqref{eq:schro}. On the other hand, if $V$ and the
$W_j$'s are polynomial of degree at most two, Gaussian initial data
lead to Gaussian solutions: if $\psi_0$ is a (complex) Gaussian, then
so is $\psi(t,\cdot)$ for all time $t\in \R$. As mentioned above,
known dynamical properties mark a difference with the linear
Schr\"odinger equation and with the nonlinear Schr\"odinger equation
in the case of power-like nonlinearities. Due to the abundance of
models involving this nonlinearity, controlling the solution
is a physically relevant question, which is interesting on the purely
mathematical level given the specificities of this equation.

\subsection{Main results}

In this article, we focus on logarithmic Schrödinger equations because
they have at least two advantages (typically compared to polynomial nonlinearities):
\begin{itemize}[topsep=0cm, parsep=0cm, itemsep=0cm]
\item They are directly well-posed in $L^2$ in any
space dimensions (on $\R^d$, polynomial nonlinearities $\pm
|\psi|^{2\sigma}\psi$ lead to global well-posedness in $L^2$ provided
that $0<\sigma<2/d$, an exponent which is sharp and obviously depends
on the dimension, see e.g. \cite{C-2003}), 
\item Their solution is globally Lipschitzian with respect to the initial condition $\psi_0\in L^2$, independently of the control.
\end{itemize}
The goal of this article is to prove the following two results.

\begin{theorem}\label{Thm:Main_torus}
Let $d \in \N^*$ and $\lambda \in \R$. If $V \in L^\infty(\T^d,\R)$ then system \eqref{logNLS_tore} is small-time $L^2$-approximately controllable.
\end{theorem}

\begin{theorem}\label{Thm:Main}
Let $d \in \N^*$ and $\lambda \in \R$. If $V$ satisfies \eqref{Hyp:V_transp} then system \eqref{logNLS+} is
small-time $L^2$-approximately controllable. 
\end{theorem}

To motivate the use of the potential $e^{-|x|^2/2}$ in \eqref{logNLS+}, we prove negative results for the following systems
\begin{equation}\label{logNLS_fin}
\left( i\partial_t+ \frac{1}{2} \Delta - V \right)\psi
=\lambda \psi \log|\psi|^2
+\left( \sum\limits_{j=1}^d u_{j}(t)x_j
\right)\psi, \quad(t,x) \in (0,T) \times \R^d, 
\end{equation}
\begin{equation}\label{logNLS_x2}
 \left( i\partial_t+ \frac{1}{2} \Delta - V\right)\psi
=\lambda \psi \log|\psi|^2
+\left( u_0(t) \frac{|x|^2}{2} + \sum\limits_{j=1}^d u_{j}(t)x_j
\right)\psi, \quad(t,x) \in (0,T) \times \R^d.
\end{equation}

\begin{theorem} \label{Thm:negatif}
\begin{enumerate}
    \item Let $d \in \N^*$ and $\lambda \in \R$. If $V(x)=\alpha |x|^2+ \beta \cdot x  + \gamma$ with $\alpha,\gamma \in \R$, $\beta \in \R^d$, then system \eqref{logNLS_x2} is not (large-time) $L^2$-approximately controllable.
    \item Let $d \in \N^*$ and $\lambda \in \R$. If $V$ satisfies \eqref{Hyp:V_transp}, then system \eqref{logNLS_fin} is not small-time $L^2$-approximately controllable. 
\end{enumerate}
\end{theorem}
 The large-time approximate controllability of system
 \eqref{logNLS_fin}, when $V$ is not quadratic, is an open
 question. It is known to hold for its linear version (i.e. for
 $\lambda=0$), generically with respect to $V$ \cite{MS-generic}.

\subsection{Proof strategy} \label{sec:Strategy}

The strategy to prove Theorems~\ref{Thm:Main_torus}  and
\ref{Thm:Main} consists in adapting to the nonlinear PDEs
\eqref{logNLS_tore} and \eqref{logNLS+} the strategy used in the
linear case (i.e. $\lambda=0$) in \cite{beauchard-pozzoli2} and
recalled below. Technically, this adaptation is achieved by replacing
algebraic manipulations of semi-groups with estimates inspired by WKB analysis.

\bigskip

We use \textbf{small-time $L^2$-approximately reachable maps} to
describe  states that can be achieved by trajectories of
\eqref{logNLS+} in arbitrarily small time. The set of $L^2$-STAR maps
  forms a subsemigroup closed for the topology of the strong
  convergence (Lemma~\ref{lem:reachable-operators} below). 

\begin{definition}[$L^2$-STAR maps] \label{def:L2STAR}
A map $L: \Sph  \to \Sph $ is $L^2$-STAR  if
for every $\psi_0 \in \Sph $ and $\varepsilon>0$, there exist $T \in
[0,\varepsilon]$, $\theta \in [0,2\pi)$ and $u \in PWC(0,T)$ such that
$\| \psi(T;u,\psi_0) - e^{i\theta} L (\psi_0) \|_{L^2} < \varepsilon$. 
\end{definition}
\begin{lemma}\label{lem:reachable-operators}
The composition and the strong limit of $L^2$-STAR maps are $L^2$-STAR
maps. 
\end{lemma}
This lemma is proved in Appendix~\ref{app}. 
\begin{definition}[Vector fields and flows $\Phi_f^s$]\label{def:flow}
$\Vect(M)$ (resp. $\Vect_c(M)$) denotes the space of 
globally Lipschitzian (resp. compactly supported)
smooth vector fields on $M$. 
For $f \in \Vect(M)$, $\Phi_f^s$
denotes the flow associated with $f$ at time $s$:
for every $x_0 \in M$, $x(s)= 
\Phi_f^s(x_0)$ is the solution of
the ODE $\dot{x}(s)=f(x(s))$ associated with the initial condition $x(0)=x_0$.
\end{definition}

\begin{definition} 
For $f \in \Vect(M)$, $s\in\R$ and $P:=\Phi_f^s$, the unitary operator on $L^2(M,\C)$ associated with $P$ is defined by
\begin{equation} \label{Def:LP}
\mathcal{L}_{P}\psi =  J_P^{1/2}(\psi\circ P),
\end{equation}
where $J_P := \text{det}(DP)$ is the determinant of the Jacobian matrix $DP$ of $P$. Then $\|\mathcal{L}_{P}\psi\|_{L^2}=\|\psi\|_{L^2}.$
\end{definition}

\begin{definition}
We introduce the following small-time controllability (STC) notions,
where $\frak{G} \subset \Vect(M)$:
\begin{itemize}
\item \textbf{STC of phases}: for every $\varphi \in L^2(M,\R)$, the
  map $\psi\mapsto e^{i\varphi}\psi$ is $L^2$-STAR,
\item \textbf{STC of flows of vector fields in $\frak{G}$}: 
for every $f \in \frak{G}$ and $t \in \R$, the map $\mathcal{L}_{\Phi_f^t}$ is $L^2$-STAR.
\end{itemize}
\end{definition}

Our strategy to prove Theorems~\ref{Thm:Main_torus} and \ref{Thm:Main}
consists in applying the following criterion, proved in \cite[Theorems
17 and 18]{beauchard-pozzoli2}, independently of the dynamics. Here
$\text{Lie}(\frak{G})$ denotes the Lie algebra generated by the vector
fields in $\frak{G}$; it is a Lie subalgebra of the Lie algebra of
smooth vector fields on $M$.

\begin{theorem} \label{thm:strategy}
Let $\frak{G} \subset \Vect(M)$ be such that
\begin{equation} \label{propriete_de_G}
\forall f \in \Vect_c(M), \quad
\exists (f_n)_{n\in\N} \subset \text{Lie}(\frak{G}) \cap \Vect(M)
 \text{ such that } 
\mathcal{L}_{\Phi_f^1} 
\text{ is the strong limit of } (\mathcal{L}_{\Phi_{f_n}^1})_{n\in\N}.
\end{equation}
Then the STC of phases and flows of vector fields in $\frak{G}$ implies the $L^2$-STAC.    
\end{theorem}

To check \eqref{propriete_de_G}, we will use the following sufficient conditions, proved in \cite[Theorem 17]{beauchard-pozzoli2}.

\begin{theorem} \label{thm:gradient-algebra}
Let $M=\T^d$ or $\R^d$ and $\frak{G} \subset \Vect(M)$.
The property \eqref{propriete_de_G} holds when
\begin{itemize}
   \item $M=\T^d$ and $\frak{G}$ contains the vector fields $\nabla\sin(x_j)$, $\nabla\cos(x_j)$, $\nabla\cos(2x_j)$, $\nabla\sin(x_j)\sin(x_k)$,\\ $\nabla\sin(x_j)\cos(x_k)$, $\nabla\cos(x_j)\cos(x_k)$
    for $j,k=1,\dots,d$ with $j\neq k$. 
    \item $M=\R^d$ and $\frak{G}$ contains the vector fields $\nabla x_j$ and $\nabla( x_j^{a} e^{-|x|^2/4} )$, for $j=1,\dots,d$ and $a=0,1,2$.
\end{itemize}    
\end{theorem}

\subsection{Structure of the article}
Section~\ref{sec:WP} is dedicated to the well-posedness of 
log-NLS, and
a representation formula. In Section~\ref{sec:phases}, we prove the STC
of phases. In Section~\ref{sec:gradient-flows}, we show the STC of
flows of gradient vector fields, by proving a version of the
Trotter-Kato product formula for the log-NLS equation. This concludes
the proofs of Theorems~\ref{Thm:Main_torus} and
\ref{Thm:Main}. Finally, in
Section~\ref{sec:non-controllability} we prove Theorem \ref{Thm:negatif}. 

\bigbreak

Throughout the article, $C$ denotes a constant whose value is
unimportant and may change from line to line. The dependence of $C$
upon various parameters is emphasized, when considered as relevant.

\section{Well-posedness of the Cauchy problem}\label{sec:WP}

\subsection{A special property of the logarithmic nonlinearity}

The following estimate is a key point in this article, and has been
extended in various ways since (see e.g. \cite[Lemma~A.1]{HaOz25}). For the sake of completeness, we recall the proof of \cite[Lemme 1.1.1]{cazenave-haraux}. 

\begin{lemma}\label{lem:CH}
For every $z_1, z_2 \in \C$,
\begin{equation} \label{Cazenave}
\left| \IM \left( 
(\overline{z_2}-\overline{z_1})
\left( 
z_2 \log |z_2|^2 - z_1 \log |z_1|^2
\right) \right)\right|
\leq 2 |z_1-z_2|^2.
\end{equation}
\end{lemma}

\begin{proof}
We have the identity 
\begin{equation} \label{identite_C}
\IM \left( 
(\overline{z_2}-\overline{z_1})
\left( 
z_2 \log |z_2| - z_1 \log |z_1|
\right) \right)
= \IM \left( \overline{z_2} z_1 \right)
\left( 
 \log |z_2| -  \log |z_1|
\right) 
\end{equation}
and the estimate
\begin{equation} \label{estim_C}
\left|\IM \left( \overline{z_2} z_1 \right)\right| = \left| \frac{ \overline{z_2} z_1 - z_2 \overline{z_1} }{ 2i } \right|= \left|\frac{z_1(\overline{z_2}-\overline{z_1})+\overline{z_1}(z_1 - z_2)}{2i}\right| \leq |z_1||z_2-z_1|.
\end{equation}
Without loss of generality, one may assume $0<|z_1|\leq|z_2|$ and then
\begin{equation} \label{log_C}
\left| \log|z_2|-\log|z_1| \right|  \leq \frac{|z_2-z_1|}{|z_1|}.
\end{equation}
We obtain \eqref{Cazenave} by gathering \eqref{identite_C}, \eqref{estim_C} and \eqref{log_C}.
\end{proof}

\subsection{Well-posedness} \label{App:WP}

In this section, we prove the well-posedness of the equations \eqref{logNLS_tore} and \eqref{logNLS+} in a unified way: we work on the generic equation \eqref{eq:schro}, which corresponds to 
\begin{itemize}
\item  either $M=\T^d$, $V \in C^{\infty}(\T^d;\R)$, $m=2d$,
  $W_{2j-1}(x)=\sin \langle b_j, x \rangle$ and $W_{2j}(x)=\cos \langle
  b_j, x \rangle$,
\item or $M=\R^d$, $V \in C^{\infty}(\R^d;\R)$ is at most quadratic i.e. satisfies \eqref{Hyp:V_transp}, $m=d+1$, $W_j(x)=x_j$ for $j=1,\dots,d$, and $W_{d+1}(x)=e^{-|x|^2/2}$.
\end{itemize}
The presence of the nonlinear term $\psi \log |\psi|^2$, with
infinite derivative at 0, introduces a difficulty solved by the
regularization argument introduced in
\cite{carles-gallagher}. We consider the function space 
\begin{equation} \label{def:Sigma}
    \Sigma := 
    \left\lbrace \begin{array}{l}
    H^1(\T^d;\C)  \quad \text{ if  } M=\T^d, 
    \\
    \{ \psi \in H^1(\R^d;\C) ; |x| \psi \in L^2(\R^d) \} \quad \text{ if  } M=\R^d.
\end{array}\right.
\end{equation}
The existence result \cite[Theorem~1.1]{CHO24} (without potential) is
readily adapted to the presence of an 
external potential, by using in addition the same arguments as in
\cite{carles-ferriere}.

\begin{proposition} \label{Prop:WP}
Let $T>0$ and $u \in PWC((0,T),\R^{m})$.
\begin{enumerate}
 \item For every $\psi_0 \in \Sigma$, the Cauchy problem
   \eqref{eq:schro} has a unique solution $\psi \in 
   L^\infty([0,T],\Sigma) \cap C^0([0,T],L^2(M))$. Moreover,
   $\|\psi(t)\|_{L^2}=\|\psi_0\|_{L^2}$ and there exists $C>0$
   such that, for every $\psi_0 \in \Sigma$,  
$\| \psi \|_{L^{\infty}([0,T],\Sigma)} \le C \| \psi_0 \|_{\Sigma}$.
 \item The $\Sigma$ solution map is uniquely extended to $L^2(M)$:
  for every $\psi_0 \in L^2(M)$, the Cauchy problem \eqref{eq:schro} has a unique solution $\psi \in C^0([0,T],L^2(M))$. Moreover, $\|\psi(t)\|_{L^2}=\|\psi_0\|_{L^2}$. 
\item The solution map is Lipschitz continuous: For every $\psi_0,
  \widetilde{\psi}_0 \in L^2(M)$ and $t \in [0,T]$, 
\begin{equation} \label{Lipschitz}
\|\psi(t;u,\psi_0)-\psi(t;u,\widetilde{\psi}_0)\|_{L^2} \leq e^{2|\lambda| t} \|\psi_0-\widetilde{\psi}_0\|_{L^2}.
\end{equation}
\end{enumerate}
\end{proposition}

\begin{proof} It suffices to work with constant controls $u_j$. 

\noindent \emph{Step 1: We prove that, for every
$\psi_0 \in \Sigma$, the Cauchy problem \eqref{eq:schro} has a unique solution $\psi \in L^{\infty}((0,T);\Sigma) \cap C^0([0,T];L^2(M))$; moreover, for every $t \in [0,T]$, $\|\psi(t)\|_{L^2}=\|\psi_0\|_{L^2}$. }
For $M=\R^d$ this is a direct application to the potential
$V_u(x):=V(x)-\sum_{j=1}^m u_j W_j(x)$ of \cite[Proposition
1.3]{carles-ferriere}, whose argument we briefly recall, to show that
it readily includes the case $M=\T^d$. For $\eps>0$, let $\psi^\eps$
denote the solution to the regularized problem
\begin{equation}\label{eq:psi-eps}
  i\d_t \psi^\eps +\frac{1}{2}\Delta \psi^\eps = V\psi^\eps +\lambda
  \psi^\eps \log\(\eps +|\psi^\eps|^2\) + \sum_{j=1}^m u_j W_j
  \psi^\eps\quad ;\quad \psi^\eps_{\mid t=0}=\psi_0. 
\end{equation}
For fixed $\eps>0$, the nonlinearity is smooth (in particular, it is
locally Lipschitz continuous), and
grows more slowly than any power-like nonlinearity. Classical results
(see e.g. \cite{BGT,Ca11,C-2003,TT-2006}) imply that the above equation
has a unique, global solution $\psi^\eps \in C(\R;\Sigma)$, and the
$L^2$-norm is preserved by the flow,
\begin{equation*}
  \|\psi^\eps(t)\|_{L^2(M)} = \|\psi_0\|_{L^2(M)},\quad \forall t\in \R.
\end{equation*}
To obtain formally this identity, multiply \eqref{eq:psi-eps} by $\overline
{\psi^\eps}$, integrate over $M$, and take the imaginary part, to get,
since $V$ and $u_jW_j$ are real-valued,
\begin{equation*}
  \frac{d}{dt} \|\psi^\eps(t)\|_{L^2(M)}^2=0.
\end{equation*}
Differentiating \eqref{eq:psi-eps} with respect to $x_j$, we get
\begin{align*}
  i\d_t \d_j\psi^\eps +\frac{1}{2}\Delta \d_j \psi^\eps &= V\d_j
  \psi^\eps+\psi^\eps\d_j V +\lambda
 \d_j \psi^\eps \log\(\eps +|\psi^\eps|^2\)
  +2\frac{\psi^\eps}{\eps+|\psi^\eps} \RE \(\overline \psi^\eps\d_j
                                                          \psi^\eps\) \\
  &\quad + \sum_{\ell=1}^m u_\ell W_\ell
 \d_j  \psi^\eps+ \sum_{\ell=1}^m u_\ell \d_j W_\ell  \psi^\eps.
\end{align*}
Multiplying this equation by $\d_j\overline
{\psi^\eps}$, integrating over $M$, and taking the imaginary part, we
find, summing over $j\in\{1,\dots,d\}$,
\begin{align*}
  \frac{1}{2} \frac{d}{dt} \|\nabla \psi^\eps(t)\|_{L^2(M)}^2
  & \le \int_M
|\nabla V| |\psi^\eps| |\nabla\psi^\eps| + 2|\lambda| \int_M
\frac{|\psi^\eps|^2}{\eps+|\psi^\eps|^2} |\nabla\psi^\eps|^2 +
\sum_{\ell=1}^m \int_M |u_\ell| |\nabla W_\ell| |\psi^\eps\nabla
    \psi^\eps|\\
 & \le 
\left\| \psi^\eps \nabla V\right\|_{L^2(M)}
   \|\nabla\psi^\eps\|_{L^2(M)} +
   2|\lambda|\|\nabla\psi^\eps\|_{L^2(M)}^2 \\
  &\quad +
\sum_{\ell=1}^m \int_M |u_\ell| \|\psi^\eps \nabla
   W_\ell \|_{L^2(M)}|\|\nabla\psi^\eps\|_{L^2(M)}  ,
\end{align*}
where we have used Cauchy-Schwarz inequality for the first and last
terms on the right hand side, and a direct estimate for the second
term. When $M=\T^d$, $\nabla V$ and $\nabla W_j$ are bounded on $M$,
and Gr\"onwall lemma yields, together with the conservation of the
$L^2$-norm of $\psi^\eps$,
\begin{equation*}
  \|\nabla \psi^\eps(t)\|_{L^2(M)}\le \|\nabla
  \psi_0\|_{L^2(M)}e^{C|t|},\quad \forall t\in \R, 
\end{equation*}
for some $C$ independent of $\eps>0$. When $M=\R^d$, $\nabla V$ may
grow linearly (in $x$), and we consider the equation satisfied by
$x_j\psi^\eps$, to proceed like with $\d_j\psi^\eps$: the
multiplication by $x_j$ commutes with all the linear terms except the
Laplacian, so
\begin{align*}
  i\d_t \(x_j\psi^\eps\) +\frac{1}{2}\Delta \(x_j \psi^\eps\)
  = \d_j \psi^\eps+ V x_j \psi^\eps+\lambda
x_j \psi^\eps \log\(\eps +|\psi^\eps|^2\)
  + \sum_{\ell=1}^m u_\ell W_\ell
 x_j  \psi^\eps,
\end{align*}
and the energy estimate (starting with the multiplication by $x_j
\overline {\psi^\eps}$) yields
\begin{equation*}
   \frac{1}{2} \frac{d}{dt} \|x \psi^\eps(t)\|_{L^2(\R^d)}^2
   \le \|x \psi^\eps(t)\|_{L^2(\R^d)} \|\nabla
  \psi^\eps(t)\|_{L^2(\R^d)}\le \frac{1}{2}\( \|x
  \psi^\eps(t)\|_{L^2(\R^d)}^2 + \|\nabla
  \psi^\eps(t)\|_{L^2(\R^d)}^2\). 
\end{equation*}
Since $D^2 V$ is bounded, $|\nabla V(x)|\le C(1+|x|)$ for some
$C>0$, the Gr\"onwall lemma applied to $\|x
  \psi^\eps(t)\|_{L^2(\R^d)}^2 + \|\nabla
  \psi^\eps(t)\|_{L^2(\R^d)}^2$ yields the existence of $C$
  independent of $\eps>0$ such that
  \begin{equation*}
  \| \psi^\eps(t)\|_{\Sigma}\le \|
  \psi_0\|_{\Sigma}e^{C|t|},\quad \forall t\in \R.
\end{equation*}
Therefore, this estimate holds both for $M=\T^d$ and $M=\R^d$.

The end of the argument is then the same as in
\cite{carles-gallagher}, we simply outline the main steps. 
In view of \eqref{eq:psi-eps}, we obtain a uniform (in $\eps>0$)
estimate for $\d_t \psi^\eps$ in $L^\infty_{\rm
  loc}(\R;\Sigma^*)$. Arzela-Ascoli theorem implies that up to a
subsequence, $\psi^\eps$ converges to some $\psi\in L^\infty_{\rm
  loc}(\R;\Sigma)\cap C(\R;L^2(M))$, and it can be checked that it solves
\eqref{eq:schro}.

\noindent \emph{Step 2: Uniqueness and Lipschitz continuity in $\Sigma$.}
Uniqueness is obtained thanks to
Lemma~\ref{lem:CH}. Let $\psi,\widetilde \psi\in L^\infty_{\rm
  loc}(\R;\Sigma)\cap C(\R;L^2(M))$ be solutions of \eqref{eq:schro}. Using the
standard $L^2$-estimate recalled above, and \eqref{Cazenave}, we obtain 
\begin{equation*}\frac{d}{dt} \| \psi-\widetilde{\psi} \|_{L^2}^2
= 2 \lambda  \int_{M} \IM  \left( (\psi-\widetilde{\psi})\left( \psi \log|\psi|^2 - \widetilde{\psi} \log|\widetilde{\psi}|^2  \right) \right) dx
 \leq 4 |\lambda| \| \psi-\widetilde{\psi} \|_{L^2}^2,
\end{equation*}
thus
$\| (\psi-\widetilde{\psi})(t) \|_{L^2}^2 \leq
\|\psi_0-\widetilde{\psi}_0\|_{L^2}^2 e^{4|\lambda|t}$ 
which gives the conclusion. The same argument yields \eqref{Lipschitz}
when the initial data $\psi_0,\widetilde \psi_0\in \Sigma$. 

\noindent \emph{Step 3: Continuation of the solution map.} 
For $\psi_0 \in L^2(M)$, let $(\psi_0^n)_{n\in\N} \subset \Sigma$ be
such that $\|\psi_0-\psi_0^n\|_{L^2} \to 0$ as $n\to \infty$. By
Step~2, the sequence of solutions $\psi^n\in L^\infty_{\rm
  loc}(\R;\Sigma)$  is a Cauchy sequence of the Banach space
$C^0([0,T];L^2(M))$, thus it converges. The estimate \eqref{Lipschitz}
passes to the limit $n \to \infty$. The most delicate part of the
argument consists in verifying that the limit $u$ is indeed a solution
to \eqref{eq:schro}, in $H^{-2}(\omega)$, with $\omega=\T^d$ if
$M=\T^d$, and $\omega$ an arbitrary open set $\omega \Subset\R^d$ if
$M=\R^d$.  This is achieved by duality arguments, using the moderate
growth of the logarithm at infinity (and equivalently, its moderate
singularity 
at the origin). The details are given in the proof of
\cite[Theorem~1.1]{CHO24}, and we omit them here.
\end{proof}

\subsection{Eikonal equation}

We recall the strategy followed in order to solve the eikonal equation
associated to the Schr\"odinger equation, as presented in
\cite[Chapter~1]{CaBook2}. The following statement is different though,
as the initial phase is at most linear (it is at most quadratic in
\cite{CaBook2}), and it contains a more
quantitative aspect, as we want to ensure a lower bound for the lifespan
of smooth solutions. 

\begin{proposition} \label{Prop:eikonal}
Let $\varphi \in C^{\infty}(M)$ at most linear (i.e. $D^\alpha \varphi \in L^{\infty}(M)$ for every $\alpha \in \N^d \setminus\{0\}$) and $s^*>0$ such that $s^* \| D^2 \varphi \|_{L^{\infty}} < 1$. There exists a unique smooth solution of
\begin{equation} \label{eq:phi}
\begin{cases}
  \partial_s \phi(s,x)+\frac{1}{2} | \nabla \phi(s,x) |^2 =0,
  & (s,x) \in (0,s^*)\times  M, \\ 
   \phi(0,\cdot)=\varphi.
\end{cases}
\end{equation}
Moreover, there exists $C=C(\varphi)>0$ such that, for every $s \in [0,s^*]$,
\begin{equation} \label{estim_phi}
\| \phi(s)-\phi(0)-s\partial_s\phi(0)\|_{L^{\infty}(M)} \leq C s^2,
\qquad
\| \nabla \phi(s)-\nabla \varphi \|_{W^{1,\infty}(M)} \leq C s.
\end{equation}
\end{proposition} 

\begin{proof} 
\noindent \emph{Step 1: Analysis:} If $\phi$ is a smooth solution of \eqref{eq:phi} and $\dot{y}(s)=\nabla \phi(s,y(s))$, $y(0)=x$, then
\begin{equation*}\frac{d}{ds} \nabla \phi(s,y(s)) =
\nabla \partial_s \phi(s,y(s)) + D^2\phi(s,y(s)).\nabla \phi(s,y(s)) =
\nabla (\partial_s \phi + \frac{1}{2}|\nabla \phi|^2)(s,y(s))=0,\end{equation*}
thus $\nabla \phi(s,y(s))=\nabla \varphi(x)$ and $y(s)=x+s\nabla \varphi(x)$. 

\medskip

\noindent \emph{Step 2: Construction of a $C^{\infty}$-diffeomorphism.} For every $s \in [0,s^*]$, the map $f_s:M \rightarrow M$ defined by $f_s(x)=x+s\nabla\varphi(x)$ is one-to-one
because
\begin{equation*}f_s(x_1)=f_s(x_2) \quad \Rightarrow \quad |x_1-x_2| = s |\nabla \varphi(x_2)-\nabla \varphi(x_1)| \leq s \|D^{2} \varphi \|_{L^{\infty}} |x_1-x_2|,\end{equation*}
and, for every $x \in M$, $Df_s(x)=I+sD^2 \varphi(x)$  is invertible, thus, by the global inverse mapping theorem $f_s$ is a $C^{\infty}$-diffeomorphism of $M$.
Moreover, for every $(s,x) \in [0,s^*]\times M$, the relation
$x=f_s^{-1}(x)+s \nabla \varphi(f_s^{-1}(x))$ implies
\begin{equation} \label{rec(f)(x)-x}
|f_s^{-1}(x)-x| \leq s \|\nabla \varphi\|_{L^{\infty}}.
\end{equation}
By the implicit function theorem, the map $(s,x) \mapsto f_s^{-1}(x)$ is smooth. 

\medskip 

\noindent \emph{Step 3: Explicit resolution of \eqref{eq:phi}.} In
\eqref{eq:phi}, once $\nabla \phi$ is known, $\phi$ is recovered by
integrating $\d_s \phi$ in time. We thus define a smooth function
$\phi$ on $[0,s^*]\times M$ by 
\begin{equation} \label{phi(s,x)_explicit}
\phi(s,x):=\varphi(x)-\frac{1}{2} \int_0^s \left| \nabla \varphi \left( f_{\sigma}^{-1}(x) \right) \right|^2 d\sigma.
\end{equation}

\noindent \emph{Step 3.a: We prove that, for every $(s,x) \in [0,s^*]
  \times  M$ then $\nabla \phi(s,x)= \nabla \varphi \left( f_s^{-1}(x)
  \right)$.} Using \eqref{phi(s,x)_explicit} and the chain rule, we
obtain
\begin{equation*}
 \nabla \phi(s,x)= \nabla \varphi(x) - \int_0^s 
\left( D f_{\sigma}^{-1}(x) \right)^T \,
D^2 \varphi \left( f_{\sigma}^{-1}(x) \right) \,
\nabla \varphi \left( f_{\sigma}^{-1}(x) \right) \,
d\sigma, 
\end{equation*}
where $A^T$ denotes the transposition of a matrix $A$.
For every $(s,x) \in [0,s^*] \times  M$, $Df_s(x)=I+sD^2 \varphi(x)$ is a symmetric matrix, thus so does
$D f_s^{-1}(x) = \left( D f_s \left( f_s^{-1}(x) \right) \right)^{-1}$.
Moreover, by differentiating with respect to the variable $s$ the
relation $x=(\mathrm{Id}+s\nabla\varphi)\left(f_s^{-1}(x)\right)$, we
obtain 
\begin{equation*}
  -\nabla \varphi \left( f_s^{-1}(x) \right)
= \left( I+s D^2 \varphi \left( f_s^{-1}(x) \right) \right)
\partial_s f_s^{-1}(x)
= D f_s\left( f_s^{-1}(x) \right)\, \partial_s f_s^{-1}(x).
\end{equation*}
Therefore
\begin{equation*}
 \nabla \phi(s,x)= \nabla \varphi(x) + \int_0^s 
\left( D f_{\sigma} \left( f_{\sigma}^{-1}(x) \right) \right)^{-1} \,
D^2 \varphi \left( f_{\sigma}^{-1}(x) \right) \,
D f_{\sigma}\left( f_{\sigma}^{-1}(x) \right)\, \partial_{\sigma} f_{\sigma}^{-1}(x) \,
d\sigma. 
\end{equation*}
For every $(s,y)\in[0,s^*]\times M$, $Df_s(y)=I+s D^2 \varphi(y)$
commutes with $D^2 \varphi(y)$ thus
\begin{equation*}
 \nabla \phi(s,x)= \nabla \varphi(x) + \int_0^s  
D^2 \varphi \left( f_{\sigma}^{-1}(x) \right) \,
\partial_{\sigma} f_{\sigma}^{-1}(x) \,
d\sigma
= \nabla \varphi(x) + \int_0^s \partial_{\sigma} \nabla \varphi \left( f_{\sigma}^{-1}(x) \right) d\sigma = \nabla \varphi \left( f_s^{-1}(x) \right). 
\end{equation*}

\noindent \emph{Step 3.b: We prove that $\phi$ solves \eqref{eq:phi}.}
We deduce from \eqref{phi(s,x)_explicit} and Step 3.a that, for every
$(s,x) \in [0,s^*] \times  M$,
\begin{equation*}
  \partial_s \phi(s,x)= - \frac{1}{2} \left| \nabla \varphi \left( f_s^{-1}(x) \right) \right|^2 = - \frac{1}{2} \left| \nabla \phi(s,x) \right|^2.
\end{equation*}

\medskip

\noindent \emph{Step 4: We prove \eqref{estim_phi}.} We deduce from
\eqref{phi(s,x)_explicit} and \eqref{rec(f)(x)-x} that, for every
$(s,x) \in [0,s^*] \times  M$
\begin{align*}
  & \left| \phi(s,x)-\phi(0,x)-s \partial_s\phi(0,x) \right|
 = 
\left| \frac{1}{2} \int_0^s \left(
\left| \nabla \varphi \left( f_{\sigma}^{-1}(x) \right) \right|^2 -
\left| \nabla \varphi ( x ) \right|^2 \right) d\sigma \right|
\\
 \leq &  \| D^2 \varphi \|_{L^{\infty}} \| \nabla \varphi\|_{L^{\infty}}
\int_0^s \left| f_{\sigma}^{-1}(x) - x\right| d\sigma
 \leq s^ 2\| D^2 \varphi \|_{L^{\infty}} 
 \| \nabla \varphi\|_{L^{\infty}}^2 .
\end{align*}
We deduce from Step 3.a and \eqref{rec(f)(x)-x} that, for every $(s,x)
\in [0,s^*] \times  M$,
\begin{align*}
  \left| \nabla \phi(s,x)-\nabla \varphi(x) \right|
&= \left| \nabla \varphi \left( f_{s}^{-1}(x) \right) - \nabla
  \varphi(x) \right| \leq s\| D^2\varphi \|_{L^{\infty}} \| \nabla
  \varphi\|_{L^{\infty}} ,\\ 
   \left| D^2 \phi(s,x)-D^2 \varphi(x) \right|
 &= \left| D^2 \varphi \left( f_s^{-1}(x) \right) D f_{s}^{-1}(x) 
-D^2 \varphi(x)
\right|\\ 
& \leq   \| D^2 \varphi \|_{L^{\infty}} \left| D f_{s}^{-1}(x) - I
  \right| + \left| D^2 \varphi \left( f_s^{-1}(x) \right)-D^2
  \varphi(x) \right| \leq s C(\varphi). 
\end{align*}
Indeed, the estimate $\left| D f_s(x)-I \right| \leq s \|D^2\varphi\|_{L^\infty}$ implies
$\left| D f_{s}^{-1}(x) - I \right| = \left| \left(D f_s\left(
      f_s^{-1}(x) \right)\right)^{-1} - I \right| \leq C s$ where
$C=C(\varphi)$ does not depend on $x \in M$. 
\end{proof}

\subsection{Representation formula for the solutions of log-NLS}

In this section, we prove a representation formula for the solutions of \eqref{eq:schro}, denoted $\psi(t;u,\psi_0)$.

\begin{definition}
For $f \in \Vect(M)$,
we define
\begin{equation} \label{Def:T_f}
D(\mathcal{T}_f):=\{ \psi \in L^2(M,\C) ; \langle f,  \nabla \psi\rangle \in L^2(M,\C) \}, \qquad
\mathcal{T}_f (\psi) = \langle f , \nabla \psi \rangle + \frac{1}{2} {\rm div}(f) \psi.
\end{equation}
\end{definition}
Since $f$ is globally Lipschitz, the method of characteristics and Liouville formula show that $e^{\mathcal{T}_f} = \mathcal{L}_P$ where $P:=\phi_f^1$ (see \cite[Lemma 29]{beauchard-pozzoli2} for details).
\begin{proposition} \label{Prop:a}
Let $\varphi, s^*, \phi$  be as in Proposition \ref{Prop:eikonal} and $\tau>0$.
For every $\psi_0 \in L^2(M)$ and $s \in (0,s^*)$, then 
$\psi(\tau s ; 0 , \psi_0 e^{i \frac{\varphi}{\tau}})
=a(s,\cdot) e^{i\frac{\phi(s,\cdot)}{\tau}}$
where
\begin{equation} \label{eq:a_epsilon}
\begin{cases}
\left( i \partial_s  + \frac{\tau}{2} \Delta  - \tau V \right) a
+ i \mathcal{T}_{\nabla \phi(s)} a 
 = \tau \lambda a \log |a|^2, & (s,x) \in (0,s^*)\times M, \\
a(0,\cdot)=\psi_0. & 
\end{cases}
\end{equation}
Moreover, there exists $C=C(\varphi,\tau)>0$ such that, for every $\psi_0 \in \Sigma$, the solution of \eqref{eq:a_epsilon} satisfies
\begin{equation} \label{estim_a_Sigma}
\|a\|_{L^{\infty}((0,s^*),\Sigma)} \leq C \| \psi_0 \|_{\Sigma}.
\end{equation}
\end{proposition}

\begin{proof}
    Let $\psi_0 \in \Sph $ and $\tau>0$. To simplify notations, we
    write $\psi(t)$ instead of $\psi(t;0,\psi_0
    e^{i\frac{\varphi}{\tau}})$. We rescale the time variable: the
    function $\xi(s,x):=\psi(\tau s,x)$ solves 
\begin{equation} \label{eq:xi_epsilon}
\begin{cases}
\left( i \partial_s+ \frac{\tau}{2} \Delta - \tau V(x) \right)\xi(s,x)
=\tau \lambda \xi \log|\xi|^2(s,x),& (s,x) \in (0,s^*) \times M, \\
\xi(0,x)=\psi_0(x) e^{i \frac{\varphi(x)}{\tau}}, &  x \in M.
\end{cases}
\end{equation}
If $\psi_0$ is smooth, then the function $a(s,x):=\xi(s,x) e^{-i\frac{\phi(s,x)}{\tau}}$ is smooth and solves \eqref{eq:a_epsilon}. If $\psi_0$ is only $L^2$, the same result holds by density. 
The estimate \eqref{estim_a_Sigma} is a consequence of Statement~1 in Proposition~\ref{Prop:WP} and \eqref{estim_phi}.
\end{proof}

\section{STC of phases}\label{sec:phases}

The goal of this section is to prove the following result.

\begin{theorem} \label{thm:phase}
    The STC of phases holds for systems \eqref{logNLS_tore} and \eqref{logNLS+}.
\end{theorem}

\subsection{A key ingredient}

In the next statement, we work on the generic equation
\eqref{eq:schro}, which corresponds to either system
\eqref{logNLS_tore} (with $W_j$ bounded)  or system \eqref{logNLS+}
(with $W_j$ at most linear).

\begin{proposition} \label{Prop2}
Let  $\alpha\in \R^{m}$ and
$\varphi(x):=\sum_{j=1}^m \alpha_j W_j(x)$.
The map $\psi\mapsto e^{i\varphi}\psi$ is $L^2$-STAR.
Indeed, for every $\psi_0 \in \Sph $, one has
\begin{equation*}
  \left\| \psi \left(\tau;-\frac{\alpha}{\tau},\psi_0\right) - e^{i \varphi} \psi_0  \right\|_{L^2} \underset{\tau \to 0}{\longrightarrow} 0.
\end{equation*}
\end{proposition}

\begin{proof}
Thanks to \eqref{Lipschitz}, one may assume that $\psi_0 \in
C_0^\infty(M,\C)$. To simplify notations, we write $\psi(t,\cdot)$
instead of 
$\psi(t;-\frac{\alpha}{\tau},\psi_0)$, which is the solution to
\begin{equation*}
  \begin{cases}
\left( i\partial_t+ \frac{1}{2} \Delta - V(x) \right)\psi(t,x)
=\lambda \psi \log|\psi|^2(t,x)
-\frac{1}{\tau} \varphi(x) \psi(t,x),& (t,x) \in (0,\tau) \times M, \\
\psi(0,\cdot)=\psi_0. &  
  \end{cases}
\end{equation*}  
Rescale the time variable by considering  $\xi(s,x):= \psi( \tau s,
x)$, for $(s,x) \in (0,1) \times M$: it solves
\begin{equation*}
  \begin{cases}
\left( i \partial_s+ \frac{\tau}{2} \Delta - \tau V(x) \right)\xi(s,x)
=\tau \lambda \xi \log|\xi|^2(s,x)
- \varphi(x) \xi(s,x),& (s,x) \in (0,1) \times M, \\
\xi(0,\cdot)=\psi_0. &  
\end{cases}
\end{equation*}
The function $\xi_{\rm app}(s,x):=e^{i s \varphi(x)} \psi_0(x)$ solves
the following system on $(0,1)_s \times M_x$ 
\begin{equation*}
\begin{cases}
\left( i \partial_s+ \frac{\tau}{2} \Delta - \tau V(x) \right)\xi_{\rm app}(s,x)
=\tau \lambda \xi_{\rm app} \log|\xi_{\rm app}|^2(s,x)
- \varphi(x) \xi_{\rm app}(s,x) + \tau R(s,x), \\
\xi(0,\cdot)=\psi_0, 
\end{cases}
\end{equation*}
where the error source term is given by
\begin{equation*}R(s,x):=\left( \frac{1}{2} \Delta -  V(x) \right) \xi_{\rm app}  (s,x)  - \lambda \xi_{\rm app} \log|\xi_{\rm app}|^2(s,x) .\end{equation*}
Using \eqref{Cazenave} in a standard $L^2$ estimate (as recalled in
the proof of Proposition~\ref{Prop:WP}), we get, since $V$ and
$\varphi$ are real-valued, for every $s \in
(0,1)$, 
\begin{align*}
    \frac{d}{ds} \left\| (\xi-\xi_{\rm app})(s) \right\|_{L^2}^2
& \leq 4 \tau |\lambda| \left\| (\xi-\xi_{\rm app})(s) \right\|_{L^2}^2 + 2 \tau \| R(s) \|_{L^2} \|(\xi-\xi_{\rm app})(s) \|_{L^2}
\\ & \leq \tau \left( 4|\lambda|+1 \right) \left\| (\xi-\xi_{\rm app})(s) \right\|_{L^2}^2 + \tau \|R(s)\|_{L^2}^2,
\end{align*}
and Gr\"onwall lemma yields
\begin{equation*} \left\| (\xi-\xi_{\rm app})(s) \right\|_{L^2}^2 \leq
\tau \int_0^s e^{\tau(4|\lambda|+1)(s-\sigma)}\|R(\sigma)\|_{L^2}^2
d\sigma.\end{equation*}
Since $\psi_0 \in C_0^\infty(M,\C)$ and $\sigma \mapsto \sigma \ln(\sigma^2)$ is continuous on $[0,\infty)$, the function $R$ is continuous and compactly supported in $M$, thus there exists 
$C=C(\psi_0,V,\lambda,\varphi)>0$ such that, for every $s \in (0,1)$,
$\|R(s)\|_{L^2}^2 \leq C$. 
Finally  
\begin{equation*}\| \psi(\tau)-e^{i \varphi} \psi_0 \|_{L^2}^2 =
\| (\xi-\xi_{\rm app})(1) \|_{L^2}^2 \leq  \tau\, C\,
e^{\tau(4|\lambda|+1)} \underset{\tau \to 0}{\longrightarrow} 0.\end{equation*}
Note that invoking the density argument to go back to the case
$\psi_0\in L^2$, we lose the above rate of convergence in $\mathcal O(\tau)$. 
\end{proof}

\subsection{On $\T^d$}

\begin{proposition} \label{Prop3_tore}
System \eqref{logNLS_tore} satisfies the following property:
if $\varphi \in C^{\infty}(\T^d,\R)$ and $\psi_0 \in L^2(M,\C)$ then
\begin{equation*}
\left\| e^{-i\frac{\varphi}{\tau}} \psi( \tau^2 ; 0 ,
  e^{i\frac{\varphi}{\tau}} \psi_0 ) - e^{-\frac{i}{2}|\nabla
    \varphi|^2} \psi_0 \right\|_{L^2} \underset{\tau \to
  0}{\longrightarrow} 0.
\end{equation*}
\end{proposition}

\begin{proof}
Let $s^*>0$ be such that $s^* \| D^2\varphi \|_{L^{\infty}}<1$ then, by Proposition \ref{Prop:eikonal}, the eikonal equation \eqref{eq:phi} is well-posed on $(0,s^*) \times \T^d$. By Proposition \ref{Prop:a}, for every $\tau \in (0,s^*)$ and $\psi_0 \in L^2(\T^d,\C)$, 
\begin{equation*}
\left\| e^{-i\frac{\varphi}{\tau}} \psi( \tau^2;0,e^{i\frac{\varphi}{\tau}}\psi_0 ) - 
e^{-\frac{i}{2}|\nabla \varphi|^2} \psi_0 \right\|_{L^2} 
=
\left\| e^{-i\frac{\varphi}{\tau}} e^{i\frac{\phi(\tau)}{\tau}} a(\tau)
- 
e^{-\frac{i}{2}|\nabla \varphi|^2} \psi_0 \right\|_{L^2}
=
\|a(\tau) - e^{-i \theta(\tau)} \psi_0 \|_{L^2}
\end{equation*}
where 
\begin{equation*}
  \theta(\tau,x):=\frac{1}{\tau} \left( \phi(\tau,x)-\phi(0,x) -\tau
    \partial_s\phi(0,x) \right),
\end{equation*}
thus,
\begin{equation*}
   \left\| e^{-i\frac{\varphi}{\tau}} \psi( \tau^2;0,e^{i\frac{\varphi}{\tau}}\psi_0 ) - 
e^{-\frac{i}{2}|\nabla \varphi|^2} \psi_0 \right\|_{L^2}  \leq \| a(\tau)-a(0) \|_{L^2} + \| \theta(\tau) \|_{L^{\infty}} \|a(\tau)\|_{L^2}
\underset{\tau \to 0}{\longrightarrow} 0,
\end{equation*}
where we have used \eqref{estim_phi} and the property $a\in
C^0([0,s^*],L^2)$, which follows from Proposition~\ref{Prop:WP}. 
\end{proof}

\begin{proposition} \label{Prop:DN}
System \eqref{logNLS_tore} satisfies the following property:
for every $\varphi \in L^2(\T^d,\R)$,
the multiplication by $e^{i \varphi}$ is $L^2$-STAR.
\end{proposition}
\begin{proof}
\emph{Step 1: A density result.}
We define recursively an increasing sequence of vector spaces:
\begin{equation*}
 \mathcal{H}_0  :=\text{span}_{\R}\{1,\sin \langle b_j,x\rangle, \cos \langle b_j , x \rangle; j \in \{1,\dots,d\} \},  
\end{equation*}
and $\mathcal{H}_j$ for $j \in \N^*$ as the largest vector space whose
elements can be written as
\begin{equation*}
  \varphi_0-\sum_{k=1}^N|\nabla\varphi_k|^2 ; N\in\N,
  \varphi_0,\dots,\varphi_N\in \mathcal{H}_{j-1}.  
\end{equation*}
By Proposition~\ref{Prop2} and Lemma~\ref{lem:reachable-operators},
for every $\varphi \in \mathcal{H}_0$, the multiplication by $e^{i
  \varphi}$ is $L^2$-STAR. By Proposition~\ref{Prop3_tore} and
Lemma~\ref{lem:reachable-operators}, for every $\varphi \in 
\mathcal{H}_{\infty} := \cup_{j\in\N} \mathcal{H}_j$, the operator
$e^{i \varphi}$ is $L^2$-STAR. Moreover, the proof of
\cite[Proposition 2.6]{duca-nersesyan} shows that
$\mathcal{H}_{\infty}$ contains any trigonometric polynomial. In particular, $\mathcal{H}_{\infty}$ is dense in $L^2(\T^d,\R)$.

\medskip

\noindent \emph{Step 2: Conclusion.}
Let $\varphi \in L^2(\T^d,\R)$. There exists $(\varphi_n)_{n\in\N} \subset \mathcal{H}_{\infty}$ such that $\| \varphi_n -\varphi \|_{L^2} \rightarrow 0 $ as $n \to \infty$. Up to an extraction, one may assume that $\varphi_n \rightarrow  \varphi$ almost everywhere on $\T^d$, as $n \to \infty$. The dominated convergence theorem proves that, for every $\psi \in L^2(\T^d,\C)$,
$\| (e^{i \varphi_n}-e^{i\varphi}) \psi \|_{L^2} \rightarrow 0$ as $n \to \infty$. Finally, Step 1 and Lemma \ref{lem:reachable-operators} prove that the operator $e^{i\varphi}$ is $L^2$-STAR.
\end{proof}

\subsection{On $\R^d$}

\begin{proposition} \label{Prop3}
System \eqref{logNLS+} satisfies the following property:
for every $j \in \{1,\dots,d\}$, $\alpha\in\R$ and $\psi_0 \in \Sph $
then
\begin{equation*}
 \left\| e^{\frac{i}{\tau}\left( \frac{\alpha^2}{2} + \alpha x_j \right)}  \psi \left( \tau ; 0 , e^{- \frac{i}{\tau} \alpha x_j} \psi_0 \right) - e^{\alpha \partial_{x_j}} \psi_0 \right\|_{L^2} \underset{\tau \to 0}{\longrightarrow} 0 
\end{equation*}
\end{proposition}

\begin{proof}
For $\varphi(x):=-\alpha x_j$, the solution of the eikonal equation
\eqref{eq:phi} is $\phi(s,x) := -\alpha x_j-\frac{\alpha^2}{2}s$ for
every $(s,x) \in \R\times\R^d$. Thus, by Proposition \ref{Prop:a}, for
every $\psi_0 \in L^2(\R^d)$,
\begin{equation*}
  \left\|  e^{\frac{i}{\tau}\left( \frac{\alpha^2}{2} + \alpha x_j \right)}  \psi \left( \tau ; 0 , e^{- \frac{i}{\tau} \alpha x_j} \psi_0 \right) - 
e^{\alpha \partial_{x_j}} \psi_0 \right\|_{L^2} 
 =
\left\| (a - \widetilde{a})(1) \right\|_{L^2}, 
\end{equation*}
where $\widetilde{a}(s,x):=\psi_0(x+\alpha s e_j )$.
By \eqref{Lipschitz}, one may assume that $\psi_0\in
C_0^\infty(\R^d;\C)$. Then, 
the function $\widetilde{a}$ solves
\begin{equation} 
\begin{cases}
\left( i \partial_s  + \frac{\tau}{2} \Delta - \tau V \right) \widetilde{a} + i \nabla \phi \cdot \nabla \widetilde{a}   = \tau \lambda \widetilde{a} \log |\widetilde{a}|^2 + \tau R(s,x), & (s,x) \in (0,1)\times \R^d, \\
\widetilde{a}(0,\cdot)=\psi_0, & 
\end{cases}
\end{equation}
where the error source term is given by
\begin{equation*}
R(s,x):= \frac{1}{2} \Delta \widetilde{a} - V \widetilde{a} -
\lambda \widetilde{a} \log|\widetilde{a}|^2 .
\end{equation*}
Proceeding like in the proof of Proposition~\ref{Prop2}, we find 
\begin{equation*}
  \frac{d}{ds} \left\| (a-\widetilde{a})(s) \right\|_{L^2}^2 \leq \tau
(4|\lambda|+1) \left\| (a-\widetilde{a})(s) \right\|_{L^2}^2 +  \tau
\|R(s)\|_{L^2}^2.
\end{equation*}
In view of the expression of $\widetilde a$, there exists $C=C(\psi_0,
\alpha,V,\lambda)>0$ such that, for every $s \in [0,1]$,
$\|R(s)\|_{L^2}^2 \leq C$ and then
\begin{equation*}
 \left\| (a-\widetilde{a})(1) \right\|_{L^2}^2 \leq \tau\, C\,
e^{\tau (4|\lambda|+1)} \underset{\tau \to 0}{\longrightarrow} 0. 
\end{equation*}
Like in the proof of Proposition~\ref{Prop2}, by density and \eqref{Lipschitz}, the above $\mathcal O(\tau)$
convergence becomes an $o(1)$ convergence. 
\end{proof}

\begin{proposition} \label{Prop:DP}
System \eqref{logNLS+} satisfies the following properties:
\begin{itemize}
\item For every $j \in \{1,\dots,d\}$ and $\alpha \in \R$, the operator $e^{\alpha \partial_{x_j}}$ is $L^2$-STAR,
\item For every $\varphi \in L^2 (\R^d,\R)$, the operator
$e^{i \varphi}$ is $L^2$-STAR.
\end{itemize}
\end{proposition}

\begin{proof}
\noindent \emph{Step 1: We prove that, for every $\varphi \in
\operatorname{span}\{ x_1,\dots,x_d, e^{-|x|^2/2} \}$, the operator $e^{i \varphi}$ is $L^2$-STAR.} This is a consequence of Proposition \ref{Prop2} and Lemma \ref{lem:reachable-operators}.

\medskip

\noindent \emph{Step 2: We prove that, for every $j \in \{1,\dots,d\}$ and $\alpha \in \R$, the operator $e^{\alpha \partial_{x_j}}$ is $L^2$-STAR.} This is a consequence of Proposition \ref{Prop3} and Lemma \ref{lem:reachable-operators}.

\medskip

\noindent \emph{Step 3: We prove that, if $\varphi \in C^1(\R^d,\R)$ and $e^{i c \varphi}$ is $L^2$-STAR for every $c \in \R$ then the operator $e^{-i \partial_{x_j} \varphi}$ is $L^2$-STAR.} Let $\tau>0$. The assumption on $\varphi$, Step 2 and Lemma \ref{lem:reachable-operators} prove that the map
\begin{equation*}
  \widetilde{L}_{\tau} :=
  e^{i\frac{\varphi}{\tau}}e^{\tau\partial_{x_j}}e^{-i\frac{\varphi}{\tau}}
\end{equation*}  
is $L^2$-STAR. The method of characteristics proves that, for every
$\psi \in L^2(\R^d,\C)$,
\begin{equation*}
  \widetilde{L}_{\tau} \psi  : x \in \R^d \mapsto \psi(x+\tau e_j) e^{-i \frac{\varphi(x+\tau e_j)-\varphi(x)}{\tau} }
 \in \C.
\end{equation*}
The continuity of the translation on $L^2(\R^d,\C)$ and the dominated convergence theorem prove that, for every $\psi \in L^2(\R^d,\C)$,
$\| (\widetilde{L}_{\tau}-e^{-i\partial_{x_j}\varphi}) \psi \|_{L^2} \to 0$ as $\tau \to 0$. Finally, by Lemma \ref{lem:reachable-operators}, 
the map $e^{-i \partial_{x_j} \varphi}$ is $L^2$-STAR.

\medskip

\noindent \emph{Step 4: Iteration.} We define recursively an increasing sequence
of vector spaces by
\begin{equation*}
  \mathcal{H}_0 : =\text{span}\{ e^{-|x|^2/2} \} \quad \text{ and } \quad
\mathcal{H}_j := \text{span}_{\R} \left\{  \varphi_0-\sum_{k=1}^d \partial_{x_k} \varphi_k ;  \varphi_0,\dots,\varphi_d \in \mathcal{H}_{j-1} \right\} \text{ for } j \in \N^*.
\end{equation*}
Thanks to Lemma \ref{lem:reachable-operators}, Steps 1 and 3,
for every $\varphi \in \mathcal{H}_{\infty} := \cup_{j\in\N} \mathcal{H}_j$, the operator $e^{i \varphi}$ is $L^2$-STAR. Moreover, by the proof of \cite[Lemma 5.2]{duca-pozzoli}, $\mathcal{H}_{\infty}$ is dense in $L^2(\R^d,\C)$ because it contains the linear combinations of Hermite functions.

\medskip

\noindent \emph{Step 5: Conclusion.} Let $\varphi \in L^2 (\R^d,\R)$. There exists $(\varphi_n)_{n\in\N} \subset \mathcal{H}_{\infty}$ such that $\| \varphi_n -\varphi \|_{L^2} \rightarrow 0 $ as $n \to \infty$. Up to an extraction, one may assume that $\varphi_n \rightarrow  \varphi$ almost everywhere on $\R^d$, as $n \to \infty$. The dominated convergence theorem proves that, for every $\psi \in L^2(\R^d,\C)$,
$\| (e^{i \varphi_n}-e^{i\varphi}) \psi \|_{L^2} \rightarrow 0$ as $n \to \infty$. Finally, Step 4 and Lemma \ref{lem:reachable-operators} prove that the operator $e^{i\varphi}$ is $L^2$-STAR.
\end{proof}

\section{STC of flows of gradient vector fields}\label{sec:gradient-flows}

\subsection{A key ingredient}

We denote by $\mathcal{R}(t;\tau,\phi)\psi_0$ the solution at time $t$
(when well-defined) of the equation 
\begin{equation}\label{logNLS+transp}
\begin{cases}
\left( i\partial_t+ \frac{\tau}{2} \Delta - \tau V(x) + i \mathcal{T}_{\nabla \phi} \right)\psi(t,x)
=\tau \lambda \psi \log|\psi|^2(t,x),& (t,x) \in (0,T) \times M, \\
\psi(0,\cdot)=\psi_0. &  
\end{cases}
\end{equation}
Note that $\psi_0 \mapsto \mathcal{R}(t;\tau,\phi) \psi_0$ is a nonlinear map. For instance the solution $\psi$ of \eqref{eq:schro} with $u=0$ is $\psi(t)=\mathcal{R}(t;1,0)\psi_0$ and the solution $a$ of \eqref{eq:a_epsilon} is $a(s)=\mathcal{R}(s;\tau,\phi) \psi_0$.

\begin{proposition} \label{Prop:Trotter-Kato}
For $\varphi \in C^{\infty}(M,\R)$ at most linear, $\tau>0$ and
$\psi_0 \in L^2(M,\C)$ then
\begin{equation*}
  \left\|
(\frak{B}_n)^n \psi_0
-
\mathcal{R}\left( 1 ; \tau ,  \varphi \right) \psi_0
\right\|_{L^2} \underset{n \to \infty}{\longrightarrow} 0
\quad \text{ where } \quad 
\frak{B}_n := e^{i \frac{|\nabla \varphi|^2}{2 n\tau}  }
e^{-i  \frac{\varphi}{\tau} }
\mathcal{R}\left(\frac{\tau}{n};1,0\right)
e^{-i \frac{\varphi}{\tau} }.
\end{equation*}
\end{proposition}

\begin{proof}
In the linear case ($\lambda=0$), this result is the celebrated
Trotter-Kato formula (see e.g. \cite{kato}). In the present nonlinear
case, this result is also reminiscent of error estimates for splitting
methods in numerical analysis (see e.g. \cite{HNW93}), from which we
borrow the scheme of the proof.   
  
Let $\varphi \in C^{\infty}(M,\R)$ at most linear, and $\tau>0$.

\noindent \emph{Step 1: We prove that, for every $n \in \N^*$, the map $\frak{B}_n$ is $e^{\frac{2\tau|\lambda|}{n}}$-Lipschitz on $L^2(M,\C)$.} 
Using \eqref{Lipschitz}, we obtain,
for every $\psi_0, \widetilde{\psi}_0 \in L^2(M,\C)$,
\begin{equation*}
 \left\| \frak{B}_n \psi_0 - \frak{B}_n \widetilde{\psi}_0 \right\|_{L^2}  =
\left\| \mathcal{R}\left(\frac{\tau}{n};1,0\right)
e^{-i \frac{\varphi}{\tau} } \psi_0 -
\mathcal{R}\left(\frac{\tau}{n};1,0\right)
e^{-i \frac{\varphi}{\tau} } \widetilde{\psi}_0
\right\|_{L^2}
 \leq 
e^{\frac{2 \tau |\lambda|}{n}} \| \psi_0 - \widetilde{\psi}_0\|_{L^2}. 
\end{equation*}

\medskip

\noindent \emph{Step 2: Reformulation of $\frak{B}_n \psi_0$ thanks to the eikonal equation.} Let $n^*$ be such that $\frac{1}{n^*} \| D^2\varphi \|_{L^{\infty}} < 1$. Then the solution $\phi$ of the eikonal equation \eqref{eq:phi} is well defined on $[0,\frac{1}{n^*}] \times M$. Let $\tau>0$. By Proposition \ref{Prop:a}, for every $n \geq n^*$ and $\psi_0 \in \Sph $,
\begin{equation*}
\frak{B}_n
\psi_0
=  e^{\frac{i}{\tau}\left(\phi(\frac{1}{n})-\phi(0)-\frac{1}{n} \partial_s \phi(0)\right)} a\left(\frac 1n \right)
= e^{\frac{i}{\tau}\left(\phi(\frac{1}{n})-\phi(0)-\frac{1}{n} \partial_s \phi(0)\right)} \mathcal{R}\left(\frac{1}{n} ; \tau , \phi\right) \psi_0,
\end{equation*}
where $a$ solves \eqref{eq:a_epsilon}.

\medskip

\noindent \emph{Step 3: We prove there exists $C>0$ such that, for every $n \geq n^*$ and $\psi_0 \in \Sigma$,
\begin{equation} \label{step1}
\left\| \frak{B}_n \psi_0
-\mathcal{R}\left(\frac{1}{n} ; \tau , \varphi \right) \psi_0
\right\|_{L^2} \leq  \frac{C}{n^{2}} \|\psi_0\|_{\Sigma}.
\end{equation}} 
Let $\widetilde{a}(s):=\mathcal{R}(s;\tau,\varphi) \psi_0$:
it solves the autonomous system
\begin{equation} \label{eq:a_tilde}
\begin{cases}
\left( i\partial_s+ \frac{\tau}{2} \Delta - \tau V(x) + i \mathcal{T}_{\nabla \varphi} \right)\widetilde{a}(s,x)
=\tau \lambda \widetilde{a} \log|\widetilde{a}|^2(t,x),& (t,x) \in (0,T) \times \R^d, \\
\psi(0,\cdot)=\psi_0. &  
\end{cases}
\end{equation}
Using \eqref{estim_phi}, we get
\begin{equation*}
  \left\| e^{\frac{i}{\tau}\left(\phi(\frac{1}{n})-\phi(0)-\frac{1}{n} \partial_s \phi(0)\right)} - 1 \right\|_{L^{\infty}}
\leq 
\frac{1}{\tau} \left\|
  \phi\left(\frac{1}{n}\right)-\phi(0)-\frac{1}{n} \partial_s \phi(0)
\right\|_{L^{\infty}} 
\leq \frac{C}{n^2}, 
\end{equation*}
where $C=C(\tau,\varphi)>0$.
By Step 2 and the previous estimate,
\begin{equation} \label{interm1}
\begin{aligned}
\left\| \frak{B}_n \psi_0
-\mathcal{R}\left(\frac{1}{n} ; \tau , \varphi \right) \psi_0
\right\|_{L^2}
& =
\left\| e^{\frac{i}{\tau}\left(\phi(\frac{1}{n})-\phi(0)-\frac{1}{n} \partial_s \phi(0)\right)} a\left(\frac{1}{n}\right) - \widetilde{a}\left(\frac{1}{n}\right) \right\|_{L^2} 
\\ & 
\leq
\frac{C\|\psi_0\|_{L^2}}{n^2}  + 
\left\| (a - \widetilde{a})\left(\frac{1}{n}\right) \right\|_{L^2}.
\end{aligned}
\end{equation}
For every $s \in (0,s^*)$, using \eqref{estim_phi} and
\eqref{estim_a_Sigma}, we obtain
\begin{equation*}
 \| \mathcal{T}_{\nabla\phi(s)-\nabla\varphi} a(s) \|_{L^2}
\leq C \|\nabla \phi(s)-\nabla \varphi\|_{W^{1,\infty}} \|a(s)\|_{\Sigma}
\leq C s \|\psi_0\|_{\Sigma}. 
\end{equation*}
The  standard energy estimate yields, using \eqref{Cazenave},
\begin{equation*}
  \frac{d}{ds}\|a-\widetilde a\|_{L^2}^2\le 2\left\|
    \mathcal{T}_{\nabla\phi(s)}a
    -\mathcal{T}_{\nabla\varphi}a\right\|_{L^2}\|a-\widetilde a\|_{L^2} +
  4|\lambda| \tau\|a-\widetilde a\|_{L^2}^2.
\end{equation*}
Gr\"onwall lemma implies
\begin{align*} 
  \left\| (a- \widetilde{a})\left(\frac 1n \right) \right\|_{L^2}
  &\le
C \int_0^{1/n} \left\| \mathcal{T}_{\nabla \phi(s)} a(s)-
  \mathcal{T}_{\nabla\varphi}a(s)\right\|_{L^2}ds=
C \int_0^{1/n} \left\| \mathcal{T}_{\nabla\phi(s)-\nabla\phi(0)} a(s)
    \right\|_{L^2} ds \\
  &\leq \frac{C}{n^2} \|\psi_0\|_{\Sigma}. 
\end{align*}
Together with \eqref{interm1}, this inequality proves the claim.

\medskip

\noindent \emph{Step 4: We prove the convergence for $\psi_0 \in \Sigma$, thanks to a telescopic argument.}
To simplify notations, we write
$\frak{A}_n:=\mathcal{R}\left(\frac{1}{n};\tau,\varphi \right)$. Then,
using Step 1, Step 3 and \eqref{estim_a_Sigma}, we obtain, by the
triangle inequality called Lady Windermere's fan in \cite{HNW93},
\begin{equation*}
\begin{aligned}
 \left\|
\left(  \frak{B}_n  \right)^{n} \psi_0
- 
\mathcal{R}\left(1;\tau,\varphi \right) \psi_0
\right\|_{L^2}
 & =  
\left\| 
(\frak{B}_n)^{n}  \psi_0 - 
(\frak{A}_n)^n \psi_0  \right\|_{L^2}
\\  & \le   \sum_{k=0}^{n-1} \left\|  
(\frak{B}_n)^k 
\frak{B}_n
(\frak{A}_n)^{n-1-k} \psi_0 
-
(\frak{B}_n)^k 
\frak{A}_n
(\frak{A}_n)^{n-1-k}
\psi_0      \right\|_{L^2}\\
&   \le 
\sum_{k=0}^{n-1} e^{2\tau |\lambda| \frac{k}{n}} 
\left\|  \frak{B}_n
(\frak{A}_n)^{n-1-k} \psi_0 
-\frak{A}_n
(\frak{A}_n)^{n-1-k}
  \psi_0      \right\|_{L^2}\\
 &   \le 
\sum_{k=0}^{n-1} e^{2\tau |\lambda| \frac{k}{n}} 
\frac{C}{n^{2}}\left\|(\frak{A}_n)^{n-1-k}
  \psi_0      \right\|_{\Sigma} 
\\  & \le 
\sum_{k=0}^{n-1} e^{2\tau |\lambda| \frac{k}{n}} 
\frac{C}{n^{2}} \|\psi_0\|_{\Sigma}
 \leq  \frac{C}{ n^{2} } \|\psi_0\|_{\Sigma} .
\end{aligned}
\end{equation*}

\noindent \emph{Step 5: We prove the convergence for every $\psi_0 \in
  L^2(M,\C)$.} By Step 1 and \eqref{Lipschitz}, for every $n \in
\N^*$, the map 
$\psi_0 \mapsto \left(  \frak{B}_n  \right)^{n} \psi_0
- 
\mathcal{R}\left(1;\tau,\varphi \right) \psi_0$
is $2 e^{2\tau |\lambda]}$-Lipschitz on $L^2(M,\C)$.
We conclude thanks to the density of $\Sigma$ in $L^2(M,\C)$.
\end{proof}

\begin{proposition} \label{Prop4}
For $\varphi \in C^{\infty} (M,\R)$ at most linear
and 
$\psi_0 \in L^2(M,\C)$ then
\begin{equation*}
  \left\|
\mathcal{R}\left( 1 ; \tau , \varphi \right) \psi_0
- e^{-\mathcal{T}_{\nabla \varphi}} \psi_0
\right\|_{L^2} \underset{\tau \to 0}{\longrightarrow} 0.
\end{equation*}
\end{proposition}

\begin{proof}
By definition, for every $\tau>0$ and $s \in [0,1]$,
$\mathcal{R}\left( s ; \tau , \varphi \right) \psi_0 =
\widetilde{a}(s)$ where $\widetilde{a}$ solves \eqref{eq:a_tilde}. Let
$b(s):=e^{-s \mathcal{T}_{\nabla \varphi}} \psi_0$: it solves
\begin{equation*}
  i\d_s b +i\mathcal{T}_{\nabla \varphi}b=0,\quad b_{\mid s=0}=\psi_0,
\end{equation*}
which can be rewritten as
\begin{equation*}
  \begin{cases}
    i\d_s b +\frac{\tau}{2}\Delta b -\tau V b+i\mathcal{T}_{\nabla
      \varphi}b=\tau \lambda  b\log|b|^2 +\tau R,\\
    b_{\mid s=0}=\psi_0, 
  \end{cases}
\end{equation*}
where
\begin{equation*}
  R =  \frac{\tau}{2}\Delta b -\tau V b -\lambda  b\log|b|^2.
\end{equation*}
For $\psi_0\in
C_0^\infty(M)$, the $L^2$ estimate with \eqref{Cazenave} yields
\begin{equation*}
   \| \widetilde{a}(s) - b(s) \|_{L^2} \leq
  C \tau \|R\|_{L^2}\le C\tau, \quad \forall s \in [0,1].
\end{equation*}
For $\psi_0 \in L^2(M,\C)$ we conclude using density and Lipschitz continuity.
\end{proof}

\subsection{On $\T^d$}

\begin{proposition} \label{Prop:STCflows_tore}
System \eqref{logNLS_tore} satisfies the following property:
for every $\varphi \in C^{\infty}(\T^d,\R)$, the operator $e^{-\mathcal{T}_{\nabla \varphi}}$ is $L^2$-STAR.
\end{proposition}

\begin{proof}
By Lemma \ref{lem:reachable-operators}, one may assume 
$\| D^2 \varphi \|_{L^2}<1$ because for every $n \in \N$,
$e^{n\mathcal{T}_{\nabla \varphi}} = \left( e^{\mathcal{T}_{\nabla \varphi}} \right)^n$. Then one may use $s^*=1$ in Propositions \ref{Prop:eikonal} and \ref{Prop:a}. By Proposition \ref{Prop:DN} and Lemma \ref{lem:reachable-operators}, 
for every $\tau>0$, the map $(\frak{B}_n)^n$ is $L^2$-approximately reachable in time $\tau^+$. By Proposition \ref{Prop:Trotter-Kato} and Lemma \ref{lem:reachable-operators}, for every $\tau$, the map $\mathcal{R}(1;\tau,\varphi)$ is $L^2$-approximately reachable in time $\tau^+$. By Proposition \ref{Prop4} and Lemma \ref{lem:reachable-operators}, $e^{-\mathcal{T}_{\nabla \varphi}}$ is $L^2$-STAR.
\end{proof}

\subsection{On $\R^d$}

\begin{proposition}\label{Prop:STCflows_Rd}
System \eqref{logNLS+} satisfies the following property:
for every $f \in \{ \alpha  e_j, \alpha \nabla( x_j^{a} e^{-|x|^2/4} ) ; \alpha \in\R, j \in \{1,\dots,d\}, a \in \{0,1,2\} \}$,
the operator $e^{\mathcal{T}_{f}}$ is $L^2$-STAR.
\end{proposition}

\begin{proof}
If $f=\alpha e_j$ then Proposition \ref{Prop:DP} gives the conclusion.
If $f=\alpha \nabla( x_j^{a} e^{-|x|^2/4} )$, the proof is the same as in the previous section.
\end{proof}

\section{A negative result}\label{sec:non-controllability}

To conclude, we prove Theorem~\ref{Thm:negatif}: we treat the first
statement in Section~\ref{subsec:V=0} and the second one in
Section~\ref{subsec:V_non_nul}. 

\subsection{Quadratic potential: invariance of Gaussian states}
\label{subsec:V=0}

In this section, we consider the system \eqref{logNLS_x2} with a
potential of the form $V(x)=\alpha |x|^2+\beta \cdot x + \gamma$ with
$\alpha,\gamma \in \R, \beta \in \R^d$. Up to a change of controls
$u_0 \leftarrow u_0+2\alpha$, $u_j \leftarrow u_j+\beta_j$, and global
phase $\psi \leftarrow \psi e^{-i\gamma t}$, one may assume that
$V=0$. The proof of the corresponding statement in
Theorem~\ref{Thm:negatif} relies on the invariance, by the dynamics,
of the set of normalized Gaussian states 
\begin{equation} \label{Gauss}
\mathcal{G}:=\{ x \in \R^d \mapsto  e^{-z_2|x|^2+z_1\cdot x+z_0} ;
z_2, z_0 \in \C, z_1 \in \C^d   \}\cap\Sph .
\end{equation}
(The intersection with the $L^2$-sphere $\Sph$ implicitly gives
$\RE(z_2)>0$ and a relation between $z_0, z_1, z_2$.)

\begin{lemma} \label{Lem:explicit}
We assume $V=0$.
Let $(a_0, b_0, c_0) \in \C \times \mathbb{C}^d \times \mathbb{C}$ be
such that $\RE(a_0)>0$ and $u=(u_0,\dots,u_{d}) \in
L^{\infty}(\R,\R^{d+1})$. The solution of \eqref{logNLS_x2} associated with the initial condition $\psi_0(x):=e^{-a_0|x|^2/2+b_0 \cdot x+c_0}$ satisfies, for every $(t,x) \in \R \times \R^d$,
$\psi(t,x)=e^{-a(t)|x|^2/2+b(t) \cdot x +c(t)}$ where
\begin{equation} \label{eq:a}
   i \dot{a}(t)=a(t)^2+2\lambda \RE(a(t))-u_0(t),   \quad
     a(0)=a_0,
\end{equation}
\begin{equation} \label{eq:b}
    i\dot{b}(t)=a(t)b(t)+2\lambda \RE(b(t))+u(t),    \quad
     b(0)=b_0,
\end{equation}
\begin{equation} \label{eq:c}
i\dot c(t) =\frac{d}{2}a(t)-\frac{1}{2}\sum_{j=1}^d b_j^2+2\lambda\RE
(c(t)) ,\quad c(0)=c_0.
\end{equation}
\end{lemma}

\begin{remark}
The maximal solution of the nonlinear system \eqref{eq:a} is defined on $\R$ (see \cite[Lemma 4.1]{carles-dong}).
Thus the maximal solution of the non-homogeneous linear system \eqref{eq:b} is defined on $\R$ too. The same property holds for \eqref{eq:c}. Moreover, we deduce from \eqref{eq:a} that, for every $t \in \R$, $\RE(a(t))=\RE(a_0) e^{-4 \int_0^t \IM(a(s)) ds} >0$, thus
$e^{a(t)|x|^2/2+b(t) \cdot x + c(t)} \in L^2(\R^d,\C)$.
\end{remark}

\begin{proof}[Proof of Lemma \ref{Lem:explicit}]
As noticed in \cite{BiMy76} in the absence of potential, and in
\cite{carles-ferriere} for the present case, Gaussian initial data
lead to Gaussian solutions in \eqref{logNLS_x2}. This is readily checked by seeking the solution to \eqref{logNLS_x2} under the form
$\psi(t,x):=e^{\phi(t,x)}$, with $\phi(t,x):=-a(t)\frac{|x|^2}{2}+b(t)
\cdot x + c(t)$. Plugging this expression into \eqref{logNLS_x2} yields the equivalent equation
\begin{equation} \label{eq_phase}
    i \partial_{t} \phi + \frac{1}{2} \Delta \phi+\frac{1}{2}
    \sum_{j=1}^d (\partial_{x_j} \phi)^2 =2 \lambda \RE(\phi)
    +u_0(t)\frac{|x|^2}{2} + u(t)
    \cdot x, \qquad \phi(0,x) =-a_0\frac{|x|^2}{2}+b_0
\cdot x +c_0.
\end{equation}
Canceling the coefficients of the above polynomial in $x$ yields 
the system  \eqref{eq:a}-\eqref{eq:b}-\eqref{eq:c}, which therefore is
equivalent to \eqref{logNLS_x2}.
\end{proof}

\begin{proof}[Proof of Theorem \ref{Thm:negatif} when $V=0$:]
$\mathcal{G}$ is a strict closed subset of $(\Sph ,\|.\|_{L^2})$ (see \cite[Lemmas 29 \& 30]{beauchard-pozzoli}). By Lemma \ref{Lem:explicit}, if $\psi_0 \in \mathcal{G}$, then the reachable set from $\psi_0$ is contained in $\mathcal{G}$, thus it is not dense in $(\Sph ,\|.\|_{L^2})$. 
\end{proof}

\subsection{General case: approximation}
\label{subsec:V_non_nul}

\begin{proposition} \label{prop:dist}
Let $\psi_0(x):=(2\pi)^{-d/4} e^{-|x|^2/2}$.
There exists $C>0$ such that, for every $u=(u_1,\dots,u_d) \in PWC(0,1)$, the solution of $\eqref{logNLS_fin}$ with initial condition $\psi_0$ satisfies, for every $t \in [0,1]$,
$\text{dist}_{L^2}(\psi(t),\mathcal{G}) \leq C t$. 
\end{proposition}

\begin{proof}
The proof relies on a representation formula for $\psi$ and an error estimate that holds uniformly with respect to the control. 
For $u=(u_1,\dots,u_d) \in PWC(\R_+,\R^d)$, the solution of \eqref{logNLS_fin} satisfies
\begin{equation*}
\psi(t,x)=\xi(t,y:=x-q(t)) e^{i \left( p(t) \cdot x  + \theta(t)\right)} ,
\end{equation*}
where
\begin{equation*}
  \begin{array}{ll}
    \left\lbrace
    \begin{array}{l}
\dot{q}(t)=p(t), \\
\dot{p}(t)=-\nabla V(q(t))-u(t), \\
(p,q)(0)=(0,0),
\end{array}\right.
\qquad
&
\qquad
\begin{array}{l}
\theta(t):=\int_0^t \left(  q(s) \cdot \nabla V(q(s)) - V(q(s))-\frac{1}{2}|p(s)|^2 \right) ds,
\\
W(t,y):=V(y+q(t))-V(q(t))- \nabla V(q(t)) \cdot y,
\end{array}
\end{array}
\end{equation*}
\begin{equation*}
  \left\lbrace \begin{array}{l}
\left(i \partial_t + \frac{1}{2} \Delta \right)\xi =W(t,y) \xi + \lambda \xi \log|\xi|^2, \qquad (t,y) \in \R_+ \times \R^d, \\
\xi(0,y)=\psi_0(y).
               \end{array}\right.
\end{equation*}             
Note that $(q,p)(t)$ is defined for every $t \in \R_+$ because $\nabla V$ is globally Lipschitz. We introduce the solution $\Gamma$ of 
\begin{equation*}
  \left\lbrace \begin{array}{l}
\left(i \partial_t + \frac{1}{2} \Delta \right)\Gamma = \lambda \Gamma \log|\Gamma|^2, \qquad (t,y) \in \R_+ \times \R^d, \\
\Gamma(0,y)=\psi_0(y).
               \end{array}\right.
 \end{equation*}
By Lemma \ref{Lem:explicit}, there exist functions $a,b,c$ (that do not depend on $u$) such that 
\begin{equation*}
  \Gamma(t,y)=e^{-a(t)|y|^2/2+ b(t) \cdot y+c(t)}.
\end{equation*}
Using Lemma \ref{lem:CH}, we obtain
\begin{equation*}
  \frac{d}{dt} \| (\xi-\Gamma)(t) \|_{L^2}^2 \leq 4|\lambda| \|
  (\xi-\Gamma)(t) \|_{L^2}^2 + 2\| W(t) \Gamma(t) \|_{L^2}
 \| (\xi-\Gamma)(t) \|_{L^2}.
\end{equation*}
Using Taylor formula, we get
\begin{equation*}
\| W(t) \Gamma(t) \|_{L^2}
\leq \| D^2 V \|_{L^\infty} \|\, |y|^2 \Gamma(t)\, \|_{L^2}.
\end{equation*}
Thus, there exists $C>0$ (that depends on $V,a,b,c,\psi_0$, but does not depend on $u$) such that, for every $t \in [0,1]$,
$\| W(t) \Gamma(t) \|_{L^2} \leq C$. Finally, for every $t\in[0,1]$,
\begin{equation*}
  \| (\xi-\Gamma)(t) \|_{L^2} \leq \int_0^t \| W(s) \Gamma(s)
  \|_{L^2} e^{4|\lambda|s} ds \leq C e^{4|\lambda|}  t.
\end{equation*}
In conclusion, the function
$\widetilde{\Gamma}_{u}:(t,x) \mapsto \Gamma(t,x-q(t)) e^{i(p(t)\cdot x+\theta(t))}$ belongs to $\mathcal{G}$ thus, for every $t \in [0,1]$,
\begin{equation*}
  \text{dist}_{L^2}(\psi(t;u,\psi_0),\mathcal{G}) \leq 
\|\psi(t)-\widetilde{\Gamma}_u(t)\|_{L^2} =
\|\xi(t)-\Gamma(t)\|_{L^2} \leq
C e^{4|\lambda|}  t.
\end{equation*}
\end{proof}

The technique used above corresponds to the coherent states approximation in semiclassical analysis (see e.g. \cite{RoCo21,CaBook2}).

\begin{proof}[Proof of Theorem \ref{Thm:negatif} in the general case:] $\mathcal{G}$ is a strict closed subset of $(\Sph ,\|.\|_{L^2})$ (see \cite[Lemmas 29 \& 30]{beauchard-pozzoli}) thus there exists $\psi_f \in \Sph$ such that $\delta:=\text{dist}_{L^2}(\psi_f,\mathcal{G})>0$. Let $\psi_0$ and $C$ be as in Proposition \ref{prop:dist}. Then, for every $T\leq\delta/2C$ and $u \in PWC((0,T),\R^d)$, we have $\|\psi(t;u,\psi_0)-\psi_f\|_{L^2}>\delta/2$. Thus, the reachable set from $\psi_0$ in any time smaller than $(\delta/2C)^2$ is not dense in $(\Sph,\|.\|_{L^2})$.
\end{proof}

\appendix

\section{Appendix} \label{app}
 
\begin{proof}[Proof of Lemma~\ref{lem:reachable-operators}]
\emph{Step 1: Semi-group structure.}
Let $L_1, L_2$ be  $L^2$-STAR maps. Let $\psi_0 \in \Sph $ and $\varepsilon>0$. 
There exist $T_2 \in [0,\varepsilon/2]$, 
$\theta_2 \in \R$ and $u_2 \in PWC(0,T_2)$ such that
\begin{equation} \label{u2}
\| \psi(T_2; u_2, L_1 \psi_0 ) -  e^{i\theta_2} L_2 L_1 \psi_0 \|_{L^2}< \frac{\varepsilon}{2}.
\end{equation}
By $\|.\|_{L^2}$-continuity of the map $\psi(T_2;u_2,\cdot)$, there exists $\delta>0$ such that, for every $\xi \in \Sph $  
\begin{equation} \label{continuite}
\|\xi-L_1 \psi_0\|_{L^2}<\delta
\qquad \Rightarrow \qquad 
\|\psi(T_2;u_2,\xi_0)-\psi(T_2;u_2,L_1 \psi_0)\|_{L^2} < \frac{\varepsilon}{2}.
\end{equation}
There exist $T_1 \in [0,\varepsilon/2]$, 
$\theta_1 \in \R$ and $u_1 \in PWC(0,T_2)$ such that
\begin{equation*}
  \| \psi(T_1;u_1,\psi_0)- e^{i\theta_1} L_1 \psi_0 \|_{L^2} < \delta.
\end{equation*}
By applying \eqref{continuite} with $\xi =
e^{-i\theta_1} \psi(T_1;u_1,\psi_0)$ and the property
$\psi(\cdot;\cdot,e^{i\theta} \cdot)=e^{i\theta} \psi(\cdot;\cdot, \cdot)$,
we obtain
\begin{equation} \label{continuite_cor}
\|\psi(T_2;u_2, \psi(T_1;u_1,\psi_0)) - e^{i\theta_1}\psi(T_2;u_2,L_1 \psi_0)\|_{L^2} < \frac{\varepsilon}{2}.
\end{equation}
We consider the time $T:=T_1+T_2$ that belongs to $[0,\varepsilon]$ and the control $u$ given by the concatenation of $u_1$ and $u_2$:
$u:=u_1 \sharp u_2 \in PWC(0,T)$. Then, using the triangular
inequality, \eqref{u2} and \eqref{continuite_cor}, we obtain 
\begin{equation*}
\begin{aligned}
& \| \psi(T;u,\psi_0) -  e^{i(\theta_1+\theta_2)} L_2 L_1 \psi_0 \|_{L^2} 
\\ \leq & \| \psi(T_2;u_2,\psi(T_1;u_1,\psi_0))- e^{i \theta_1} \psi(T_2;u_2,  L_1 \psi_0) \|_{L^2} + \| \psi(T_2;u_2, L_1 \psi_0)-   e^{i\theta_2} L_2 L_1 \psi_0\|_{L^2}
  \leq  \varepsilon.
\end{aligned}
\end{equation*}

\noindent \emph{Step 2: Stability by strong convergence.} Let $(L_n)_{n\in\N}$ be a sequence of $L^2$-STAR maps and $L:\Sph  \to \Sph $. We assume that $(L_n)_{n\in\N}$ strongly converges towards $L$, i.e. for every $\psi \in L^2(M,\C)$, $\|L_n(\psi)-L(\psi)\|_{L^2} \to 0$ as $n \to \infty$.

 Let $\psi_0 \in \Sph $ and $\varepsilon>0$. There exists $n \in \N$ such that 
$\|L_n(\psi_0) -L(\psi_0)\|_{L^2} < \varepsilon/2$. There exists $T \in [0,\eps]$, $\theta \in \R$ and $u\in \text{PWC}(0,T)$ such that 
$\| \psi(T;u,\psi_0)- e^{i\theta} L_n \psi_0 \|_{L^2} < \varepsilon/2$. Then 
$
\| \psi(T;u,\psi_0)- e^{i\theta} L \psi_0 \|_{L^2}
\leq \| \psi(T;u,\psi_0)- e^{i\theta} L_n \psi_0 \|_{L^2} + 
\| L_n(\psi_0) -L(\psi_0)\|_{L^2} < \varepsilon.
$
\end{proof}

\textbf{Acknowledgments.} The authors are supported by the Centre Henri Lebesgue, program ANR-11-
LABX-0020. Karine Beauchard acknowledges support from grant
4 (project TRECOS) and from the Fondation Simone et Cino
Del Duca – Institut de France. Eugenio Pozzoli acknowledges support from
grants ANR-24-CE40-3008-01 (project QuBiCCS). This project has received
financial support from the CNRS through the MITI interdisciplinary programs.

\bibliographystyle{abbrv}
\bibliography{references}



\end{document}